\documentclass[11pt,english]{article}
\usepackage[T1]{fontenc}
\usepackage[latin9]{inputenc}
\usepackage{geometry}
\usepackage{color}
\usepackage{babel}
\usepackage{refstyle}
\usepackage{mathrsfs}
\usepackage{amsmath}
\usepackage{amsthm}
\usepackage{amssymb}
\usepackage[all]{xy}
\usepackage[unicode=true,pdfusetitle,
 bookmarks=true,bookmarksnumbered=false,bookmarksopen=false,
 breaklinks=false,pdfborder={0 0 1},backref=false,colorlinks=false]
 {hyperref}

\makeatletter

\AtBeginDocument{\providecommand\eqref[1]{\ref{eq:#1}}}
\RS@ifundefined{subsecref}
  {\newref{subsec}{name = \RSsectxt}}
  {}
\RS@ifundefined{thmref}
  {\def\RSthmtxt{theorem~}\newref{thm}{name = \RSthmtxt}}
  {}
\RS@ifundefined{lemref}
  {\def\RSlemtxt{lemma~}\newref{lem}{name = \RSlemtxt}}
  {}

\numberwithin{equation}{section}
\numberwithin{figure}{section}
\theoremstyle{plain}
\newtheorem{thm}{\protect\theoremname}[section]
\theoremstyle{definition}
\newtheorem{defn}[thm]{\protect\definitionname}
\theoremstyle{remark}
\newtheorem{rem}[thm]{\protect\remarkname}
\theoremstyle{plain}
\newtheorem{lem}[thm]{\protect\lemmaname}
\theoremstyle{plain}
\newtheorem{prop}[thm]{\protect\propositionname}
\theoremstyle{plain}
\newtheorem{assumption}[thm]{\protect\assumptionname}
\theoremstyle{plain}
\newtheorem{cor}[thm]{\protect\corollaryname}
\theoremstyle{definition}
\newtheorem{example}[thm]{\protect\examplename}
\theoremstyle{remark}
\newtheorem{notation}[thm]{\protect\notationname}

\usepackage{amssymb}

\makeatother

\providecommand{\assumptionname}{Assumption}
\providecommand{\corollaryname}{Corollary}
\providecommand{\definitionname}{Definition}
\providecommand{\examplename}{Example}
\providecommand{\lemmaname}{Lemma}
\providecommand{\notationname}{Notation}
\providecommand{\propositionname}{Proposition}
\providecommand{\remarkname}{Remark}
\providecommand{\theoremname}{Theorem}

\begin{document}
\global\long\def\RR{\mathbb{R}}%
\global\long\def\CC{\mathbb{C}}%
\global\long\def\HH{\mathbb{H}}%
\global\long\def\ZZ{\mathbb{Z}}%
\global\long\def\QQ{\mathbb{Q}}%

\global\long\def\gam{\Gamma}%

\global\long\def\e{\epsilon}%
\global\long\def\a{\alpha}%
\global\long\def\b{\beta}%
\global\long\def\ga{\gamma}%
\global\long\def\ph{\varphi}%
\global\long\def\om{\omega}%
\global\long\def\dl{\delta}%
\global\long\def\lm{\lambda}%

\global\long\def\exd#1#2{\underset{#2}{\underbrace{#1}}}%
\global\long\def\exup#1#2{\overset{#2}{\overbrace{#1}}}%

\global\long\def\leb#1{\operatorname{Leb(#1)}}%

\global\long\def\norm#1{\left\Vert #1\right\Vert }%

\global\long\def\porsmall{\prec}%
\global\long\def\porbig{\succ}%

\global\long\def\diffeo{\simeq}%

\global\long\def\del{\partial}%
\global\long\def\dprime{\prime\prime}%

\global\long\def\manifold{\mathcal{M}}%
\global\long\def\mtx{M}%

\global\long\def\sbgrp{H}%

\global\long\def\transpose{\mbox{t}}%

\global\long\def\norm#1{\Vert#1\Vert}%
\global\long\def\bignorm#1{\left\Vert #1\right\Vert }%

\global\long\def\brac#1{(#1)}%
\global\long\def\vbrac#1{|#1|}%
\global\long\def\sbrac#1{[#1]}%
\global\long\def\dbrac#1{\langle#1\rangle}%
\global\long\def\cbrac#1{\{#1\}}%

\global\long\def\gl#1{\operatorname{GL}_{#1}}%

\global\long\def\sl#1{\operatorname{SL}_{#1}}%

\global\long\def\so#1{\operatorname{SO}_{#1}}%

\global\long\def\ort#1{\operatorname{O}_{#1}}%

\global\long\def\pgl#1{\operatorname{PGL}_{#1}}%

\global\long\def\po#1{\operatorname{PO}_{#1}}%

\global\long\def\wc{Q}%

\global\long\def\interior#1{\operatorname{int}\left(#1\right)}%

\global\long\def\dist#1{\operatorname{dist}\brac{#1}}%

\global\long\def\diag#1{\operatorname{diag}\brac{#1}}%

\global\long\def\rank#1{\operatorname{rank}\left(#1\right)}%

\global\long\def\covol#1{\operatorname{covol}\brac{#1}}%

\global\long\def\sym{\mbox{Sym}}%

\global\long\def\sp#1#2{\mbox{span}_{#1}\brac{#2}}%

\global\long\def\id{\operatorname{id}}%

\global\long\def\idmat#1{\operatorname{I}_{#1}}%

\global\long\def\Ad#1{\operatorname{Ad}_{#1}}%

\global\long\def\comp#1{\operatorname{#1}^{c}}%

\global\long\def\lieG{\mathfrak{g}}%
\global\long\def\lieA{\mathfrak{a}}%
\global\long\def\lieN{\mathfrak{n}}%

\global\long\def\lieNelement{Z}%
\global\long\def\dimN{p}%

\global\long\def\topindex{q}%

\global\long\def\nbhd#1#2{\mathcal{O}_{#1}^{#2}}%

\global\long\def\symfund#1{F_{#1}}%
\global\long\def\groupfund#1{\widetilde{F_{#1}}}%

\global\long\def\Gset{\mathcal{B}}%
\global\long\def\symset{\mathcal{E}}%

\global\long\def\ball#1{B_{#1}}%

\global\long\def\sphere#1{\mathbb{S}^{#1}}%

\global\long\def\gras#1{\operatorname{Gr}(#1)}%

\global\long\def\latspace#1{\mathcal{L}_{#1}}%

\global\long\def\shapespace#1{\mathcal{X}_{#1}}%

\global\long\def\pairspace#1{\mathcal{P}_{#1}}%

\global\long\def\roundo{r}%

\global\long\def\fam#1#2{#1_{#2}}%

\global\long\def\smeq{\overset{_{\text{s.m.}}}{\simeq}}%

\global\long\def\lat{\Lambda}%
\global\long\def\latfull{\lat_{0}}%

\global\long\def\perpen#1{#1^{\perp}}%

\global\long\def\factor#1{#1^{\pi}}%

\global\long\def\doublearrowdown{\rotatebox[origin=c]{315}{\ensuremath{\longleftrightarrow}}}%
\global\long\def\doublearrowup{\rotatebox[origin=c]{45}{\ensuremath{\longleftrightarrow}}}%

\global\long\def\volmin{V_{\text{min}}}%
\global\long\def\volmax{V_{\text{max}}}%
\global\long\def\radmax{R}%

\global\long\def\fdomN{\mathscr{D}}%
\global\long\def\domN{{\cal D}}%

\global\long\def\minset{\Delta\symset_{T}}%

\title{A practical guide to well roundedness}
\author{Tal Horesh\thanks{IST Austria, \texttt{tal.horesh@ist.ac.at}; supported by EPRSC grant EP/P026710/1.} 
\and Yakov Karasik\thanks{Institut für  Algebra und Geometrie, KIT, Karlsruhe, Germany \texttt{theyakov@gmail.com}.}}
\maketitle
\begin{abstract}
Let $G$ be a semisimple algebraic group. We develop a machinery for
manipulation and manufacture of well-rounded families $\left\{ \Gset_{T}\right\} _{T>0}\subset G$
as they were defined in a work by A.\ Gorodnik and A.\ Nevo. The
importance of these types of families is that one can asymptotically
count lattice points in them and even obtain an error term. Lattice
counting is highly effective for solving asymptotic problems from
number theory and the geometry of numbers.

The tools we develop are handy especially when the family is given
w.r.t.\ some decomposition of $G$ (e.g. Iwasawa or Cartan) and also
when it depends upon a sub-quotients of the form $\mathcal{M}/H$,
where $\mathcal{M}\subset G$ is a submanifold and $H<G$ is a closed
subgroup. 
\end{abstract}

\section{Introduction}

This text came to life out from the authors' work \cite{HK_gcd,HK_dlattices,HK_flags}
on equidistrbution problems in geometry of numbers. In the course
of our work we relied significantly on counting lattice point results
in semisimple algebraic groups. The classical setting for counting
lattice points problems is $\RR^{d}$, where in order to establish
an asymptotic formula for the number of lattice points inside an increasing
family of compact subsets of $\RR^{d}$, it is required that the boundary
of the sets satisfies certain regularity conditions (e.g. smoothness,
finite non vanishing curvature, etc.). With the rise of interest in
counting problems in semisimple groups and their affine spaces, e.g.
hyperbolic spaces (e.g. \cite{EM93}), where the boundary of a set
has fundamentally different properties than in $\RR^{d}$, a new notion
of regularity for a family $\left\{ \Gset_{T}\right\} $ was born:
\emph{well roundedness}. It was introduced in \cite{EM93} (although
a similar idea appeared already in \cite{DRS93}), was used e.g. in
\cite{GW07}, and then defined again in \cite{GN1}, for the general
setting of lcsc groups. In the latter, the concept of well roundedness
was further refined to Lipschitz well roundedness: 
\begin{defn}[\cite{GN1}]
\label{def: well--roundedness}Let $G$ be a locally compact second
countable group with a Borel measure $\mu$, and let $\left\{ \nbhd{\e}{}\right\} _{\e>0}$
be a family of identity neighborhoods in $G$. Assume $\left\{ \Gset_{T}\right\} _{T>0}\subset G$
is a family of measurable domains and denote 
\[
\Gset_{T}^{\left(+\e\right)}:=\nbhd{\e}{}\Gset_{T}\nbhd{\e}{}=\bigcup_{u,v\in\nbhd{\e}{}}u\,\Gset_{T}\,v,
\]
\[
\Gset_{T}^{\left(-\e\right)}:=\bigcap_{u,v\in\nbhd{\e}{}}u\,\Gset_{T}\,v.
\]
 The family $\left\{ \Gset_{T}\right\} $ is \emph{Lipschitz well-rounded
(LWR)} with (positive) parameters $\left(\mathcal{C},T_{0},\e_{0}\right)$
if for every $0<\e<\e_{0}$ and $T>T_{0}$: 
\begin{equation}
\mu\left(\Gset_{T}^{\left(+\e\right)}\right)\leq\left(1+\mathcal{C}\e\right)\:\mu\left(\Gset_{T}^{\left(-\e\right)}\right).\label{eq:LWReq}
\end{equation}
The parameter $\mathcal{C}$ is called the \emph{Lipschitz constant}
of the family $\left\{ \Gset_{T}\right\} $. 
\end{defn}

The results in \cite{GN1} show that under some conditions on $G$,
which originate from representation theory, one can get an estimate
with an error term of the size of the set $\Gamma\cap\text{\ensuremath{\Gset_{T}}}$,
where $\Gamma<G$ is a lattice and $T\to\infty$. To be able to use
these results to solve concrete problems, one should have an ample
supply of LWR sets as well as a mechanism for checking that a certain
family is indeed well rounded. This is the first aim of this work. 

Despite their name, simple (and semisimple) groups are quite complicated
objects, and they are often studied via their decompositions into
``less complicated'' groups. A decomposition is when a group (or
some ``big'' subset of it) is written as a product of certain subgroups,
e.g. the Cartan, Iwasawa and Bruhat decompositions. In the context
of counting lattice points, many natural counting problems translate
into counting lattice points in families of increasing sets inside
semisimple Lie groups, that are defined via a decomposition of the
group. As a baby example, consider the Cartan decomposition $KAK$
of $\so{1,n}^{0}\left(\RR\right)$, with $K=\so n\left(\RR\right)$
and $A=\left\{ a_{t}:t\in\RR\right\} $, $a_{t}=\left[\begin{smallmatrix}\cosh t & 0 & \sinh t\\
0 & I_{n-2} & 0\\
\sinh t & 0 & \cosh t
\end{smallmatrix}\right]$. The set $\Gset_{T}=\left\{ Ka_{t}K:0\leq t\leq T\right\} $ is the
lift of a hyperbollic ball of radius $T$ in $\HH^{n}\simeq K\backslash\so{1,n}^{0}\left(\RR\right)$,
and so counting lattice points in the family $\left\{ \Gset_{T}\right\} _{T>0}$
is essentially equivalent to the hyperbolic sphere problem \cite{Lax_Phillips,Phillips_Rudnick},
first stated by Selberg, which concerns counting lattice orbit points
in increasing hyperbolic balls. We refer to \cite{GN1} (see also
\cite{GOS_wavefront}) for counting in sets that are defined via the
Cartan decomposition, to \cite{MMO14} for counting in sets that are
defined via the Bruhat decomposition, and to \cite{HN16_Counting}
for counting in sets that are defined via the Iwasawa decomposition,
which is also the setting in our aforementioned ongoing work, for
which this text was written. 

When verifying the well roundedness of a family of sets inside a semisimple
Lie group, it is therefore quite natural to consider a suitable decomposition
of the group, and then try and reduce the verification to the that
of well roundedness of the projections of the family in each one of
the subgroups that appear in the decomposition. The logic being that
these subgroups are easier to analyze since they are compact, or abelian,
or unipotent, etc. In the example with the lifts of hyperbolic balls,
one would like to reduce the well roundedness of the family $\left\{ \Gset_{T}\right\} _{T>0}$
to well roundedness of $\left\{ a_{t}:0\leq t\leq T\right\} _{T>0}$
inside the subgroup $A$ (and of the constant family $\left\{ K\right\} $
inside $K$, which is trivial). However, it is false that well roundedness
in the components of the decomposition implies well roundedness of
the original family in the group, so this reduction cannot be implemented
without  further thought. We have considered a systematic approach
to this problem, which can be also applied in other situation as well.
The idea is a categorical approach of defining morphisms between groups,
called \textbf{roundomorphisms}, which pull back a well rounded family
in the image into a well rounded family in the domain.

When it comes to counting problems in a semisimple group, many families
of interest have the property that their projections to one or more
of the components is a fixed set. A simple example is the the projection
to the $K$ components of the lifted hyperbolic balls, but in fact
such families arise naturally in equidistribution problems (see \cite{Truelsen,HN16_Counting}
and our aforementioned work in progress). Since our method relies
on well roundedness in the components, it is helpful to formulate
a condition that is easier to verify than the one of well roundedness
itself, yet implies well roundedness for a constant family of sets;
by that we mean $\Gset_{T}=\Gset$ for all $T$. We propose the following:
\begin{defn}
Let $\manifold$ be an orbifold. A subset $B$ of $\manifold$ is
called a Boundary Controllable Set, or BCS, if for every $x\in\mathcal{M}$
there is an open neighborhood $U_{x}$ of $x$ such that $U_{x}\cap\del B$
is contained in a finite union of embedded submanifolds of $\mathcal{M}$.
\end{defn}

\begin{rem}
\label{rem:GlobalBCS}In most cases (e.g. when $\mathcal{M}$ is a
manifold and $B$ is bounded) the global version of the previous definition
is sufficient. That means one can take for every point $x\in\mathcal{M}$
the open set $\mathcal{M}$, and so it is sufficient to check that
the boundary of $B$ is contained in a finite number of embedded submanifolds
with dimension strictly smaller than $\dim\left(\manifold\right)$.
\end{rem}

Orbifolds arise naturally as a quotient of a submanifold of $G$ by
a group with almost free stabilizers. One example would be the space
$\so n\left(\RR\right)\backslash\sl n\left(\RR\right)/\sl n\left(\ZZ\right)$
which bears great significance in geometry of numbers, being the space
of \emph{shapes} of lattices, and when $n=2$ it is the modular curve.
Our aforementioned work in progress concerns equidistribution in spaces
of lattices that have the structure of an orbifold, e.g. the space
of shapes. Since our approach is lifting the counting problem from
the quotient to the group $G$, where we count in well rounded families,
we are forced to transfer BCS's from the quotient space into some
convenient fundamental domain inside the group. It turns out that
not every fundamental domain is adequate, and in Section \ref{sec: spread models}
we characterize the fundamental domains that are. We call them \textbf{spread
models}, since we think of them as fundamental domains that are obtained
by cutting the space open and spreading it - here one can think of
cutting a two dimensional torus into a parallelogram. We hope that
this section could be of further interest in the future, since it
is essentially a discussion on fundamental domains for which the quotient
map into the space the domain represents, pulls back differential
properties from the space to the domain. This investigation is the
second main goal of this work. 

Finally in the last section \ref{subsec: Siegel reduced bases}, we
provide some examples for spread models. In particular, we introduce
the well known (\cite{Grenier_93,Schmidt_98}) constructions of fundamental
domains coming from geometry of numbers and show that they are indeed
spread models for the spaces of lattices that they represent.

\section{Coordinate balls}

Definition \ref{def: well--roundedness} is w.r.t.\ a nested family
$\left\{ \nbhd{\e}{}\right\} _{\e>0}$ of identity neighborhoods in
the group, where by ``nested'' we mean that $\e_{1}<\e_{2}$ implies
$\nbhd{\e_{1}}{}\subset\nbhd{\e_{2}}{}$. While the definition of
well roundedness  allows any nested family of identity neighborhoods,
we shall work only with neighborhoods that are the images of small
balls in the Lie algebra under the exponent map --- this is Assumption
\ref{assu: Our O_e }, which concludes the current subsection. The
advantages of this choice follow from the fact that it is a special
case of \emph{coordinate balls} (Definition \ref{def: coordinate balls}),
and this subsection is devoted to investigating the properties of
neighborhoods of this sort. 
\begin{defn}[Equivalence of identity neighborhoods]
\label{def: equivalence of nbhds}Let $G$ be a Lie group and consider
two families $\left\{ \nbhd{\e}{}\right\} _{\e>0},\left\{ \nbhd{\e}{\prime}\right\} _{\e>0}$
of nested and symmetric identity neighborhoods. We say that these
families are \emph{equivalent} if there exist $\e_{1},c,C>0$ such
that for every $0<\e<\e_{1}$
\[
\nbhd{c\e}{}\subseteq\nbhd{\e}{\prime}\subseteq\nbhd{C\e}{}.
\]
\end{defn}

\begin{defn}[Coordinate balls]
\label{def: coordinate balls}A family $\left\{ \nbhd{\e}{}\right\} _{\e>0}$
of identity neighborhoods inside a Lie group $G$ will be called a
family of \emph{coordinate balls} if there exist a ball $\ball{\e}=\left\{ x\in\RR^{\dim\left(G\right)}:\norm x<\e\right\} $
inside $\RR^{\dim\left(G\right)}$, and a $C^{1}$ chart 
\[
\phi:\exd U{1_{G}\in}\to\RR^{m}
\]
of the identity, such that $\left\{ \phi^{-1}\left(\ball{\e}\right)\right\} _{\e>0}$
is equivalent to $\left\{ \nbhd{\e}{}\right\} _{\e>0}$. 
\end{defn}

\begin{rem}
\label{rem: coordinate balls equivalent}All coordinate balls of a
given Lie group are equivalent. Indeed, if $\phi_{1}$ and $\phi_{2}$
are two charts, then $\phi_{2}\phi_{1}^{-1}|_{\ball 1}$ is a \emph{bi-Lipschitz
map. }Hence, \emph{
\[
\phi_{2}^{-1}\left(\ball{c\e}\right)\subseteq\underset{\phi_{1}^{-1}\left(\ball{\e}\right)}{\underbrace{\phi_{2}^{-1}\left(\phi_{2}\phi_{1}^{-1}\left(\ball{\e}\right)\right)}}\subseteq\phi_{2}^{-1}\left(\ball{C\e}\right)
\]
 for some $c,C>0$ and $\e<1$.}
\end{rem}

The following Lemma specifies two useful features of coordinate balls. 
\begin{lem}
\label{lem: Connectivity + additivity of coord. balls}Let $\left\{ \nbhd{\e}{}\right\} _{\e>0}$
be a family of coordinate balls inside a Lie group $G$. Then for
small enough $\e$ and $\delta$, the following two properties hold:
\begin{itemize}
\item (\textbf{Connectivity}) $\nbhd{\e}{}$ is a connected subset of $G$.
\item (\textbf{Additivity}) There exists $c>0$ such that:
\[
\nbhd{\e}{}\nbhd{\dl}{}\subseteq\nbhd{c\left(\e+\delta\right)}{}.
\]
\end{itemize}
\end{lem}

\begin{proof}
Connectivity holds since $\phi^{-1}$ ($\phi$ being the associated
chart) is continuous. Additivity holds for Riemannian left $G$-invariant
balls with $c=1$ (triangle inequality); these Riemannian balls are
indeed coordinate balls, where the implied chart is the Riemannian
exponential map. Since all families of coordinate balls are equivalent
(Remark \ref{rem: coordinate balls equivalent}), the statement follows. 
\end{proof}
One last property of coordinate balls is the following. \ref{lem: Connectivity + additivity of coord. balls}
\begin{prop}
\label{prop: multplication in balls is addition in coordinates}Let
$\left\{ \nbhd{\e}{}\right\} _{\e>0}$ be a family of coordinate balls
inside a Lie group $G$, and assume 
\[
\phi:\exd U{g\in}\to\RR^{m}
\]
is a chart that contains an element $g$. Then, there exist an open
set $g\in V\subset U$ and positive $\e\left(g\right),c\left(g\right)$
such that for $\e\leq\e\left(g\right)$:
\[
\nbhd{\e}{}V\nbhd{\e}{}\subseteq U
\]
and for every $h\in\overline{V}$
\[
\phi\left(\nbhd{\e}{}h\nbhd{\e}{}\right)\subseteq\phi\left(h\right)+\ball{c\left(g\right)\e}.
\]
\end{prop}

The proof requires an auxiliary lemma:
\begin{lem}[\cite{HN16_Counting}]
\label{lem: Conjugation inflates by the norm of Ad} Let $G$ be
a Lie group with Lie algebra $\lieG$. For $\nbhd{\e}{}=\exp\left(\ball{\e}\right)$
 and every $g\in G$, 
\[
g^{-1}\,\nbhd{\e}{}\,g\subseteq\nbhd{\e\cdot\left\Vert \Ad g\right\Vert _{\text{op}}}{}=\exp\left\{ \lieNelement\in\lieG:\left\Vert \lieNelement\right\Vert \leq\e\cdot\left\Vert \Ad g\right\Vert _{\text{op}}\right\} ,
\]
where $\norm{\cdot}$ is any euclidean norm on $\lieG$ and $\left\Vert \cdot\right\Vert _{\mbox{op}}$
is the norm on the space of linear operators on $\lieG$. 
\end{lem}

\begin{proof}[Proof of Proposition \ref{prop: multplication in balls is addition in coordinates}]
 Observe that by the previous lemma and the additivity property
in Lemma \ref{lem: Connectivity + additivity of coord. balls}, for
every $h\in G$ there is a constant $c_{1}\left(h\right)$ such that
for $0<\e\leq\e_{1}\left(h\right)$:
\[
\phi\left(\nbhd{\e}{}h\nbhd{\e}{}\right)=\phi\left(h\cdot\underset{\subseteq\nbhd{\e\cdot\left\Vert \Ad h\right\Vert _{\text{op}}}{}}{\underbrace{h^{-1}\nbhd{\e}{}h}}\cdot\nbhd{\e}{}\right)\subseteq\phi\left(h\cdot\nbhd{c_{1}\left(h\right)\e}{}\right)=\underset{\psi_{h}}{\underbrace{\phi\circ L_{h}\circ\exp}}\left(\ball{c_{1}\left(h\right)\e}\right),
\]
where $L_{h}:G\to G$ is the left translation by $h$. By compactness
of $D_{g}:=\overline{\nbhd 1{}g\nbhd 1{}}$ and continuity of $\left\Vert \Ad{}\left(\cdot\right)\right\Vert _{\text{op}}$
, there exist $c_{0}\left(g\right)$ and $\e_{0}\left(g\right)$
for which the above holds uniformly on $D_{g}$, namely for every
$h\in D_{g}$ and $0<\e\leq\e_{0}\left(g\right):$
\[
\phi\left(\nbhd{\e}{}h\nbhd{\e}{}\right)\subseteq\psi_{h}\left(\ball{c_{0}\left(g\right)\e}\right).
\]
 We choose $\e\left(g\right)>0$ and $0<\delta<1$ small enough so
that (using the additivity property again) for $V:=\nbhd{\dl}{}g\nbhd{\dl}{}$
and $0<\e<\e\left(g\right)$ we have
\[
\nbhd{\e}{}V\nbhd{\e}{}=\nbhd{\e}{}\nbhd{\dl}{}g\nbhd{\dl}{}\nbhd{\e}{}\subset\nbhd{\e_{0}\left(g\right)}{}g\nbhd{\e_{0}\left(g\right)}{}.
\]
We also assume $\e_{0}\left(g\right)$ is small enough such that $\nbhd{\e_{0}\left(g\right)}{}g\nbhd{\e_{0}\left(g\right)}{}\subset U$. 

Since $\psi\left(h,x\right)=\psi_{h}\left(x\right)$ is a differentiable
map defined on a compact domain  $\overline{U}\times\overline{\ball{c_{0}\left(g\right)\e}}$,
 there exists $c\left(g\right)=c\left(D_{g}\right)>0$ such that
for every $h\in U$ and $x\in\ball{c_{0}\left(g\right)\e}$:
\[
\norm{\psi_{h}\left(x\right)-\psi_{h}\left(0\right)}\leq c\left(g\right)\norm{x-0}.
\]
Hence,
\[
\psi_{h}\left(\ball{c_{0}\left(g\right)\e}\right)\subseteq\psi_{h}\left(0\right)+\ball{c\left(g\right)\e}=\phi\left(h\right)+\ball{c\left(g\right)\e}.
\]
\end{proof}
Finally, we fix a choice of coordinate balls that will be used from
now on.
\begin{assumption}
\label{assu: Our O_e }Unless specified otherwise we will assume that
$\nbhd{\e}{}=\exp\left(\ball{\e}\right)$, where $\exp$ is the Lie
exponent.
\end{assumption}

\section{Well rounded sets - criteria and properties}

This section is devoted to investigating the concept of well-roundedness
for \emph{constant} families, which are just fixed subsets of $G$:
$\Gset_{T}=\Gset$ for all $T$. It turns out that in the constant
case, the LWR property can be reduced to a boundary condition. This
enables us to obtain a huge class of Lipschitz well rounded (fixed)
sets, which are in fact the BCS's defined in the introduction. 
\begin{lem}
\label{lem:B_plus B_minus}Suppose $\left\{ \nbhd{\e}{}\right\} _{\e>0}$
is a family of coordinate balls, and let $\Gset\subseteq G$. Then
$\Gset^{+}\left(\e\right)\setminus\Gset^{-}\left(\e\right)=\nbhd{\e}{}\,\del\Gset\,\nbhd{\e}{}$,
or equivalently:
\[
\Gset^{\left(+\e\right)}=\Gset\cup\left(\nbhd{\e}{}\,\del\Gset\,\nbhd{\e}{}\right)
\]
 and 
\[
\Gset^{\left(-\e\right)}=\Gset\setminus\left(\nbhd{\e}{}\,\del\Gset\,\nbhd{\e}{}\right).
\]
\end{lem}

\begin{rem}
In fact, Lemma \ref{lem:B_plus B_minus} applies for any family $\left\{ \nbhd{\e}{}\right\} _{\e>0}$
of connected identity neighborhoods.
\end{rem}

\begin{proof}
We first show that 
\[
\Gset^{\left(+\e\right)}\setminus\Gset^{\left(-\e\right)}=\nbhd{\e}{}\,\del\Gset\,\nbhd{\e}{}.
\]
For the inclusion $\supseteq$, we must show that $\Gset^{\left(+\e\right)}\supseteq\nbhd{\e}{}\cdot\del\Gset\cdot\nbhd{\e}{}$
and that $\left(\nbhd{\e}{}\cdot\del\Gset\cdot\nbhd{\e}{}\right)\cap\Gset^{\left(-\e\right)}=\emptyset$.
For the first, assume $g\in\nbhd{\e}{}\cdot\del\Gset\cdot\nbhd{\e}{}$.
By symmetry of $\nbhd{\e}{}$, the open set $\nbhd{\e}{}\cdot g\cdot\nbhd{\e}{}$
intersects $\del\Gset$ non-trivially, and therefore meets $\Gset$,
say in a point $h$. Then (again by symmetry) $g\in\nbhd{\e}{}\cdot h\cdot\nbhd{\e}{}\subset\nbhd{\e}{}\Gset\nbhd{\e}{}$.
For the latter, note that $h\in\Gset^{\left(-\e\right)}$ if and only
if $h\in u\Gset v$ for all $u,v\in\nbhd{\e}{}$, i.e. if and only
if $u^{-1}hv^{-1}\in\Gset$ for all $u,v\in\nbhd{\e}{}$, which by
symmetry of $\nbhd{\e}{}$ is equivalent to $\nbhd{\e}{}\cdot h\cdot\nbhd{\e}{}\subset\Gset$.
Now if $g\in\nbhd{\e}{}\cdot\del\Gset\cdot\nbhd{\e}{}$ then as before
the open set $\nbhd{\e}{}\cdot g\cdot\nbhd{\e}{}$ intersects $\del\Gset$
non-trivially, and in particular meets $\comp{\Gset}$; then $\nbhd{\e}{}\cdot g\cdot\nbhd{\e}{}\not\subset\Gset$,
namely $g\notin\Gset^{\left(-\e\right)}$. 

For the inclusion $\subseteq$, let $g\notin\nbhd{\e}{}\del\Gset\nbhd{\e}{}$,
and we show that $g\notin\Gset^{\left(+\e\right)}\setminus\Gset^{\left(-\e\right)}$.
Namely, that either $g\in\Gset^{\left(-\e\right)}$ or that $g\in\comp{\brac{\Gset^{\left(+\e\right)}}}$.
Indeed, $g\notin\nbhd{\e}{}\del\Gset\nbhd{\e}{}$ implies that $\left(\nbhd{\e}{}g\nbhd{\e}{}\right)\cap\del\Gset=\emptyset$,
and since $\nbhd{\e}{}g\nbhd{\e}{}$ is connected it follows that
either $\nbhd{\e}{}g\nbhd{\e}{}\subseteq\Gset$ or $\nbhd{\e}{}g\nbhd{\e}{}\subseteq\comp{\Gset}$.
The first implies (by the equivalence established in the first inclusion)
that $g\in\Gset^{\left(-\e\right)}$. The latter implies that  $g\notin\nbhd{\e}{}\Gset\nbhd{\e}{}=\Gset^{\left(+\e\right)}$. 

The statement of the lemma now follows:
\[
\Gset^{\left(+\e\right)}=\Gset^{\left(-\e\right)}\sqcup\nbhd{\e}{}\del\Gset\nbhd{\e}{}\subseteq\Gset\cup\nbhd{\e}{}\del\Gset\nbhd{\e}{}
\]
where the opposite inclusion holds as  $\Gset^{\left(+\e\right)}\supseteq\nbhd{\e}{}\del\Gset\nbhd{\e}{}$.
Furthermore, 
\[
\Gset^{\left(-\e\right)}=\Gset^{\left(+\e\right)}\setminus\nbhd{\e}{}\del\Gset\nbhd{\e}{}=\left(\Gset\cup\nbhd{\e}{}\del\Gset\nbhd{\e}{}\right)\setminus\nbhd{\e}{}\del\Gset\nbhd{\e}{}=
\]
\[
=\left(\left(\Gset\setminus\nbhd{\e}{}\del\Gset\nbhd{\e}{}\right)\sqcup\nbhd{\e}{}\del\Gset\nbhd{\e}{}\right)\setminus\nbhd{\e}{}\del\Gset\nbhd{\e}{}=\Gset\setminus\nbhd{\e}{}\del\Gset\nbhd{\e}{}.
\]
\end{proof}
From Lemma \ref{lem:B_plus B_minus} we deduce the following simple
criterion for the Lipschitz well roundedness of a (fixed) set.
\begin{lem}
\label{lem: LWR for single set}Let $G$ be a Lie group with a Borel
measure $\mu$. If a subset $\Gset\subset G$ satisfies that $0<\mu\left(\Gset\right)<\infty$
and that there exists $c>0$ such that 
\[
\mu\brac{\nbhd{\e}{}\cdot\del\Gset\cdot\nbhd{\e}{}}\leq c\epsilon
\]
for every $0<\epsilon<\frac{\mu\left(\Gset\right)}{2c}$, then $\Gset$
is LWR with 
\[
C=\frac{2c}{\mu\left(\Gset\right)}.
\]
The converse also holds: suppose $\Gset$ is LWR with positive measure
and parameter $C$. Then for $\e<C^{-1}$,
\[
\mu\brac{\nbhd{\e}{}\cdot\del\Gset\cdot\nbhd{\e}{}}\leq C\mu\left(\Gset\right)\epsilon.
\]
\end{lem}

\begin{proof}
By our assumption, for $\epsilon<\frac{\mu\left(\Gset\right)}{2c}$,
and by Lemma \ref{lem:B_plus B_minus}: 

\begin{eqnarray*}
\mu\left(\Gset^{\left(+\e\right)}\right) & = & \mu\left(\nbhd{\e}{}\Gset\nbhd{\e}{}\right)\\
 & = & \mu\left(\Gset^{\left(-\e\right)}\right)+\mu\left(\nbhd{\e}{}\cdot\del\Gset\cdot\nbhd{\e}{}\right)\\
 & \leq & \mu\left(\Gset^{\left(-\e\right)}\right)+c\e
\end{eqnarray*}
and
\begin{eqnarray*}
\mu\brac{\Gset^{\left(-\e\right)}} & = & \mu\left(\Gset\setminus\left(\nbhd{\e}{}\cdot\del\Gset\cdot\nbhd{\e}{}\right)\right)\\
 & \geq & \mu\left(\Gset\right)-\mu\left(\nbhd{\e}{}\cdot\del\Gset\cdot\nbhd{\e}{}\right)\\
 & \geq & \mu\left(\Gset\right)-c\e\\
_{\left(\e<\frac{\mu\left(\Gset\right)}{2c}\right)} & \geq & \frac{\mu\left(\Gset\right)}{2}
\end{eqnarray*}
As a result, for $\e<\frac{\mu\left(\Gset\right)}{2c}$, 
\[
\frac{\mu\brac{\Gset^{\left(+\e\right)}}-\mu\brac{\Gset^{\left(-\e\right)}}}{\mu\brac{\Gset^{\left(-\e\right)}}}\leq\frac{c\e}{\frac{1}{2}\mu\left(\Gset\right)}=\frac{2c}{\mu\left(\Gset\right)}\cdot\e.
\]

Regarding the opposite direction, our assumption is that for $\e<C^{-1}$,
\[
\frac{\mu\brac{\Gset^{\left(+\e\right)}}-\mu\brac{\Gset^{\left(-\e\right)}}}{\mu\brac{\Gset^{\left(-\e\right)}}}\leq C\e.
\]
Hence, 
\[
\frac{\mu\brac{\nbhd{\e}{}\cdot\del\Gset\cdot\nbhd{\e}{}}}{\mu\left(\Gset\right)}\leq\frac{\mu\brac{\Gset^{\left(+\e\right)}}-\mu\brac{\Gset^{\left(-\e\right)}}}{\mu\brac{\Gset^{\left(-\e\right)}}}\leq C\e.
\]
In other words, 
\[
\mu\brac{\nbhd{\e}{}\cdot\del\Gset\cdot\nbhd{\e}{}}\leq\mu\left(\Gset\right)C\e.
\]
\end{proof}
One consequence of Lemma \ref{lem: LWR for single set} is that finite
unions and intersections of LWR sets are in themselves LWR. 
\begin{lem}
\label{lem: intersection and union of LWR sets}Let $G$ be a Lie
group with a Borel measure $\mu$. If two subsets $\Gset$ and $\Gset^{\prime}$
of $G$ such that $0<\mu\left(\Gset\cap\Gset^{\prime}\right)$ are
LWR, then $\Gset\cap\Gset^{\prime}$ and $\Gset\cup\Gset^{\prime}$
are also LWR with Lipschitz constant
\[
C_{\Gset\cap\Gset^{\prime}}=2\max\left\{ C,C^{\prime}\right\} \cdot\frac{\mu\left(\Gset\right)+\mu\left(\Gset^{\prime}\right)}{\mu\left(\Gset\cap\Gset^{\prime}\right)};\,\,C_{\Gset\cup\Gset^{\prime}}=2\max\left\{ C,C^{\prime}\right\} \cdot\frac{\mu\left(\Gset\right)+\mu\left(\Gset^{\prime}\right)}{\mu\left(\Gset\cup\Gset^{\prime}\right)}.
\]
\end{lem}

\begin{proof}
We prove the lemma only for the intersection $\Gset\cap\Gset^{\prime}$;
the proof for the union $\Gset\cup\Gset^{\prime}$ is similar. By
Lemma \ref{lem: LWR for single set}, for $\e<\frac{1}{C_{\Gset\cap\Gset^{\prime}}}$
(so $\e<C^{-1},C^{\prime-1}$):
\begin{eqnarray*}
\mu\left(\nbhd{\e}{}\cdot\del\Gset\cdot\nbhd{\e}{}\right) & \leq & C\:\mu\left(\Gset\right)\epsilon\,,\\
\mu\left(\nbhd{\e}{}\cdot\del\Gset^{\prime}\cdot\nbhd{\e}{}\right) & \leq & C^{\prime}\mu\left(\Gset^{\prime}\right)\epsilon.
\end{eqnarray*}

Hence, by using the fact that the boundary of an intersection is contained
in the union of the boundaries, we obtain that for $\e<\frac{1}{C_{\Gset\cap\Gset^{\prime}}}$
\begin{align*}
\mu\left(\nbhd{\e}{}\cdot\del\left(\Gset\cap\Gset^{\prime}\right)\cdot\nbhd{\e}{}\right) & \leq\mu\left(\nbhd{\e}{}\cdot\del\Gset\cdot\nbhd{\e}{}\right)+\mu\left(\nbhd{\e}{}\cdot\del\Gset^{\prime}\cdot\nbhd{\e}{}\right)\\
 & \leq\max\left\{ C,C^{\prime}\right\} \cdot\left(\mu\left(\Gset\right)+\mu\left(\Gset^{\prime}\right)\right)\cdot\e
\end{align*}
The first direction of Lemma \ref{lem: LWR for single set} yields
the desired conclusion. 
\end{proof}
Using Lemma \ref{lem: LWR for single set}, which provides us with
an if and only if criterion for Lipschitz well roundedness of a fixed
set, we will now obtain that the sets with controlled boundary are
indeed LWR. 
\begin{prop}
\label{prop: BCS is well rounded}Let $G$ be a Lie group. Assume
that $\mu$ is a measure on $G$ that is absolutely continuous w.r.t.\
Haar measure, and has density that is bounded on compact sets. If
$\Gset$ is a compact BCS with $\mu\left(\Gset\right)>0$, then $\Gset$
is Lipchitz well-rounded. 
\end{prop}

\begin{proof}
The strategy is to apply Lemma \ref{lem: LWR for single set}. This
will be done by showing that for a subset $Y$ of $G$ which is compact
and consists of a finite union of subsets of embedded submanifolds
of strictly smaller dimension (e.g. the boundary of $\Gset$) there
exist $c=c\left(Y\right),\e\left(Y\right)>0$ such that
\begin{equation}
\mu\left(\nbhd{\e}{}Y\nbhd{\e}{}\right)\leq c\epsilon\label{eq:meager}
\end{equation}
for some $0<\e<\e\left(Y\right)$.

It is clearly sufficient to assume that $Y$ is contained in one submanifold.
For each point $g\in Y$, there is some chart $\phi_{g}:U_{g}\to\RR^{m}$
for which $g\in U_{g}$ and $\phi\left(U_{g}\cap Y\right)\subseteq\RR^{m-1}\times\left\{ 0\right\} $.
Let $V_{g}$ be the open sets from Proposition \ref{prop: multplication in balls is addition in coordinates}
which satisfy: $g\in V_{g}\subseteq U_{g}$. By compactness, there
are $g_{1},\dots,g_{r}\in Y$ for which $V_{g_{1}},\dots,V_{g_{r}}$
cover $Y$ entirely. In order to establish the inequality in Formula
(\ref{eq:meager}), it is sufficient to prove it for each $Y\cap\overline{V}_{g_{i}}$
separately. Consequentially, we may assume that $r=1$: $g_{1}=g$,
$V_{g_{1}}=V$, $Y_{0}=Y\cap\overline{V}$ and $\phi_{g_{1}}=\phi$.\textcolor{black}{{} }

By Proposition \ref{prop: multplication in balls is addition in coordinates},
there exist $c\left(g\right),\e\left(g\right)>0$ such that for $\e<\e\left(g\right)$
and $h\in\overline{V}$, $\phi\left(\nbhd{\e}{}h\nbhd{\e}{}\right)\subseteq\phi\left(h\right)+\ball{c\left(g\right)\e}$.
In particular 
\[
\phi\left(\nbhd{\e}{}Y_{0}\nbhd{\e}{}\right)\subseteq\phi\left(Y_{0}\right)+\ball{c\left(g\right)\e}.
\]
Hence it is sufficient to show that $\phi_{*}\mu\left(\phi\left(Y_{0}\right)+\ball{c\left(g\right)\e}\right)\leq c\e$. 

Let $\om\in L^{1}\left(\RR^{m}\right)$ be such that $\phi_{*}\mu=\om\cdot\mu_{\RR^{m}}$
where $\mu_{\RR^{m}}$ is the Lebesgue measure on $\RR^{m}$. Then,
since $\om$ is bounded on compact sets  (and in particular on $\phi\left(\overline{\nbhd{\e\left(g\right)}{}Y_{0}\nbhd{\e\left(g\right)}{}}\right)$),
it is sufficient to show that
\[
\mu_{\RR^{m}}\left(\phi\left(Y_{0}\right)+\ball{\e}\right)\leq c\e.
\]
Indeed, since $Y_{0}$ is an embedded submanifold, there exists a
bounded set $E\subseteq\RR^{m-1}$ such that $\phi\left(Y_{0}\right)+\ball{\e}\subseteq E\times\left[-c_{2}\e,c_{2}\e\right]$,
which implies the desired result.
\end{proof}

\section{Roundomorphisms\label{sec: Roundomorphisms}}

Roughly speaking, the difficulty in checking well roundedness inside
a simple non compact Lie group arises from the fact that well roundedness
is a \emph{multiplicative} property, while simple Lie groups are ``highly
non-abelian''. Nevertheless, simple Lie groups have several known
decompositions --- Cartan, Iwasawa, etc. --- which allow them to
be written as the product of more ``convenient'' subgroups. E.g.,
in the case of the Iwasawa decomposition, the subgroups $K,A,N$ are
compact, abelian and nilpotent respectively, which makes it considerably
easier to prove well roundedness inside them. The goal of this section
is to reduce the question of whether a family $\Gset_{T}\subset G$
is LWR, to verifying LWR of the projections of $\Gset_{T}$ to each
of the components of $G$ w.r.t.\ a given decomposition. E.g. when
considering the Iwasawa decomposition, the well roundedness of $\Gset_{T}$
is reduced to the question of well roundedness of the image of $\Gset_{T}$
in the direct product $K\times A\times N$. This can be achieved \emph{if}
the Iwasawa diffeomorphism $G\to K\times A\times N$ preserves  well
roundedness; maps with this property are the topic of the following
definition. 
\begin{defn}[Roundomorphism]
\label{def: roundomorphism}Let $G$ and $\sbgrp$ be two topological
groups with measures $\mu_{G}$ and $\mu_{\sbgrp}$, and let $\left(\nbhd{\e}G\right)_{\e>0}$
and $\left(\nbhd{\e}{\sbgrp}\right)_{\e>0}$ be two families of identity
neighborhoods in $G$ and $\sbgrp$ respectively. A Borel measurable
map $\roundo:G\to\sbgrp$ will be called an \emph{$f$-roundomorphism}
if it is:
\begin{enumerate}
\item \textbf{Measure preserving:} $\roundo_{*}\left(\mu_{G}\right)=\mu_{\sbgrp}$.
\item \textbf{Locally Lipschitz:} $\roundo\left(\nbhd{\e}Gg\nbhd{\e}G\right)\subseteq\nbhd{f\e}{\sbgrp}\roundo(g)\nbhd{f\e}{\sbgrp}$
for some continuous $f=f\left(g\right):G\to\RR_{>0}$ and for every
$0<\e<\frac{1}{f}$.
\end{enumerate}
\end{defn}

The following proposition reveals the motivation for defining roundomorphisms,
as well as the reason they are called that way: they pull back LWR
families to LWR families. 
\begin{prop}
\label{prop: roundo pulls back LWR to LWR}Let $\roundo:G\to\sbgrp$
be an $f$-roundomorphism. Assume that $\left\{ \Gset_{T}\right\} _{T>0}$
is a family of measurable subsets of $\sbgrp$ such that $f$ is bounded
uniformly on $\roundo^{-1}\left(\Gset_{T}\right)$ by a constant $F$.
If $\left\{ \Gset_{T}\right\} $ is LWR with parameters $\left(T_{0},C_{0}\right)$,
then the pre-image $\roundo^{-1}\left(\Gset_{T}\right)$ is LWR with
parameters $\left(T_{0},F\cdot\max\left\{ C_{0},1\right\} \right)$.
\end{prop}

\begin{proof}
The strategy of the proof is to show that for $\e<F^{-1}$,
\begin{equation}
\mu_{G}\left(\left(\roundo^{-1}\left(\Gset_{T}\right)\right)^{\left(+\e\right)}\right)\leq\mu_{\sbgrp}\left(\Gset_{T}^{\left(+F\e\right)}\right)\label{eq: roundo1}
\end{equation}
and
\begin{equation}
\mu_{\sbgrp}\left(\Gset_{T}^{\left(-F\e\right)}\right)\leq\mu_{G}\left(\left(\roundo^{-1}\left(\Gset_{T}\right)\right)^{\left(-\e\right)}\right).\label{eq: roundo2}
\end{equation}
It will then follow that for $T>T_{0}$ and $\e<\frac{1}{F\cdot\max\left\{ C_{0},1\right\} }$
(so that both $\e<F^{-1}$ and $\e<\left(FC_{0}\right)^{-1}$: the
first for inequalities (\ref{eq: roundo1}) and (\ref{eq: roundo2})
to hold, and the second for the LWR of $\left\{ \Gset_{T}\right\} $),
\[
\frac{\mu_{G}\left(\left(\roundo^{-1}\left(\Gset_{T}\right)\right)^{\left(+\e\right)}\right)}{\mu_{G}\left(\left(\roundo^{-1}\left(\Gset_{T}\right)\right)^{\left(-\e\right)}\right)}\leq\frac{\mu_{\sbgrp}\left(\Gset_{T}^{\left(+F\e\right)}\right)}{\mu_{\sbgrp}\left(\Gset_{T}^{\left(-F\e\right)}\right)}\leq1+FC_{0}\e.
\]
Inequalities (\ref{eq: roundo1}) and (\ref{eq: roundo2}) follow
from measure preservation of $r$, along with the following inclusions:
\begin{align*}
\left(\roundo^{-1}\left(\Gset_{T}\right)\right)^{\left(+\e\right)}\subseteq & r^{-1}\left(\Gset_{T}^{\left(+F\e\right)}\right),\\
\left(\roundo^{-1}\left(\Gset_{T}\right)\right)^{\left(-\e\right)}\supseteq & r^{-1}\left(\Gset_{T}^{\left(-F\e\right)}\right),
\end{align*}
that we now justify. For the first, note that by definition of a roundomorphism,
$\nbhd{\e}Gg\nbhd{\e}G\subseteq r^{-1}\left(\nbhd{f\e}{\sbgrp}\roundo(g)\nbhd{f\e}{\sbgrp}\right)$.
Hence, $\nbhd{\e}G\cdot\roundo^{-1}\left(\Gset_{T}\right)\cdot\nbhd{\e}G\subseteq\roundo^{-1}\left(\nbhd{F\e}{\sbgrp}\Gset_{T}\nbhd{F\e}{\sbgrp}\right)$.
For the second inclusion, suppose $g\in r^{-1}\left(\Gset_{T}^{\left(-F\e\right)}\right)$.
We want to show that if $u,v\in\nbhd{\e}G$, then $ugv\in r^{-1}\left(\Gset_{T}\right)$.
Put differently, $r\left(ugv\right)\in\Gset_{T}$. This is indeed
the case, since $r\left(ugv\right)=u^{\prime}r\left(g\right)v^{\prime}$
for some $u^{\prime},v^{\prime}\in\nbhd{F\e}{\sbgrp}$ (local Lipschitzity
of $r$), and $u^{\prime}r\left(g\right)v^{\prime}\in\Gset_{T}$ since
$r\left(g\right)\in\Gset_{T}^{\left(-F\e\right)}$. 
\end{proof}
The most useful incident of Proposition \ref{prop: roundo pulls back LWR to LWR}
is when $\sbgrp$ (such that $r:G\to\sbgrp$ is a roundomorphism)
is a direct product of\textcolor{black}{{} groups. This is what allows
us to reduce (under certain conditions) well roundedness in} the group
$G$ to well roundedness in the components of a decomposition of $G$. 
\begin{cor}
\label{cor: roundo to a produt grp}Let $\roundo:G\to\sbgrp=\sbgrp_{1}\times\cdots\times\sbgrp_{\topindex}$
be an $f$-roundomorphism and let $\Gset_{T}=\Gset_{T}^{1}\times\cdots\times\Gset_{T}^{\topindex}\subseteq\sbgrp$.
Set 
\begin{enumerate}
\item $\mu_{\sbgrp}=\mu_{\sbgrp_{1}}\times\cdots\times\mu_{\sbgrp_{\topindex}}$
\item $\nbhd{\e}{\sbgrp}=\nbhd{\e}{\sbgrp_{1}}\times\cdots\times\nbhd{\e}{\sbgrp_{\topindex}}$
\end{enumerate}
and assume that: 
\begin{enumerate}
\item For $j=1,\dots,\topindex$: $\Gset_{T}^{j}\subseteq\sbgrp_{j}$ is
LWR w.r.t.\ the parameters $\left(T_{j},C_{j}\right)$;
\item $f$ is bounded uniformly by $F$ on the sets $\roundo^{-1}\left(\Gset_{T}\right)$.
\end{enumerate}
Then $\roundo^{-1}\left(\Gset_{T}\right)$ is LWR, w.r.t.\ the parameters
\[
T=\max\left\{ T_{1},\dots,T_{\topindex}\right\} ,\;C\asymp_{\topindex}F\cdot\max\left\{ C_{1},\dots,C_{\topindex},1\right\} .
\]
\end{cor}

\begin{proof}
It is sufficient to prove the claim for $\topindex=2$, where one
then proceeds by induction. According to the previous proposition
we only need to show that $\Gset_{T}$ is Lipchitz well-rounded w.r.t.\
the parameters $\left(T,C/F\right)$. Indeed, since
\[
\mu_{\sbgrp}\left(\Gset_{T}^{\left(\pm\e\right)}\right)=\mu_{\sbgrp_{1}}\left(\left(\Gset_{T}^{1}\right)^{\left(\pm\e\right)}\right)\cdot\mu_{\sbgrp_{2}}\left(\left(\Gset_{T}^{2}\right)^{\left(\pm\e\right)}\right),
\]
we obtain
\[
\frac{\mu_{\sbgrp}\left(\Gset_{T}^{\left(+\e\right)}\right)}{\mu_{\sbgrp}\left(\Gset_{T}^{\left(-\e\right)}\right)}\leq\left(1+C_{1}\e\right)\left(1+C_{2}\e\right)\leq\left(1+\max\left\{ C_{1},C_{2}\right\} \e\right)^{2}\leq1+3\max\left\{ C_{1},C_{2}\right\} \e
\]
 for $\e\porsmall\frac{1}{\max\left\{ C_{1},C_{2}\right\} }$.
\end{proof}
\begin{rem}
\label{rem: product of LWR is LWR}One consequence of Corollary \ref{cor: roundo to a produt grp}
is that a direct product of LWR families 
\[
\Gset_{T}^{1}\times\cdots\times\Gset_{T}^{\topindex}\subseteq\sbgrp_{1}\times\cdots\times\sbgrp_{\topindex}
\]
is LWR. To see this, take $G=\sbgrp_{1}\times\cdots\times\sbgrp_{\topindex}$
and $r$ that is the identity map on $G$; it is a roundomorphism
with $f\equiv1$.
\end{rem}

The content of the following lemma is that a composition of roundomorphisms
is a roundomorphism.
\begin{lem}
\label{lem: composition of roundomorphisms}Suppose that $r_{1}:G_{1}\to G_{2}$
is an $f_{1}$-roundomorpism and $r_{2}:G_{2}\to G_{3}$ is an $f_{2}$-roundomorphism.
Then, $r_{2}\circ r_{1}$ is an $f=\left(f_{2}\circ r_{1}\right)\cdot f_{1}$-roundomorphism.
\end{lem}

\begin{proof}
Clearly we only need to check that $r_{2}\circ r_{1}$ is locally
Lipchitz:
\[
r_{2}r_{1}\left(\nbhd{\e}{G_{1}}\cdot g\cdot\nbhd{\e}{G_{1}}\right)\subseteq r_{2}\left(\nbhd{f_{1}\e}{G_{2}}\cdot r_{1}\left(g\right)\cdot\nbhd{f_{1}\e}{G_{2}}\right)\subseteq\nbhd{f\e}{G_{3}}\cdot r_{2}r_{1}\left(g\right)\cdot\nbhd{f\e}{G_{3}}.
\]
\end{proof}
Finally, any smooth map from $G_{1}$ to $G_{2}$ such that $\roundo_{*}\left(\mu_{G_{1}}\right)=\mu_{\sbgrp_{2}}$
is a roundomorphism:
\begin{prop}
\label{prop: smooth map is Lipschitz}Let $G_{1}$ and $G_{2}$ be
Lie groups and $r:G_{1}\to G_{2}$ a sooth map. Then $r$ is locally
Lipschitz.
\end{prop}

\begin{proof}
Let $g\in G_{1}$ with $\phi_{1}:U\to\mathbb{R}^{n}$ a chart at $g$.
By Proposition \ref{prop: multplication in balls is addition in coordinates},
there is an open neighborhood $V\subset U$ of $g$ and $\e_{0},c_{1}>0$
such that for every $0<\e<\e_{0}$,
\[
\phi_{1}\left(\mathcal{O}_{\e}g\mathcal{O}_{\e}\right)\subseteq\phi\left(g\right)+B_{c_{1}\e}.
\]
Let $L_{r(g)^{-1}}:G_{2}\to G_{2}$ be the left translation by $r(g)^{-1}$,
and let $W$ be an open neighborhood of $1_{G_{2}}$ such that $\ln_{G_{2}}|_{W}$
is a diffeomorphism onto an open neighborhood of $\mathfrak{g}_{2}$.
We may assume that $L_{r(g)^{-1}}\circ r(U)\subseteq W$. We get that
there is $c_{2}>0$ such that for every $\e<\e_{0}$:
\[
\ln_{G_{2}}\circ L_{r(g)^{-1}}\circ r\left(\mathcal{O}_{\e}g\mathcal{O}_{\e}\right)\subseteq\ln_{G_{2}}\circ L_{r(g)^{-1}}\circ r\circ\phi^{-1}\left(\phi\left(g\right)+B_{c_{1}\e}\right)\subseteq B_{c_{2}\e}.
\]
As a result,
\[
r\left(\mathcal{O}_{\e}g\mathcal{O}_{\e}\right)\subseteq r(g)\cdot\exp_{G_{2}}\left(B_{c_{2}\e}\right).
\]
\end{proof}

\section{\label{sec:Well-roundedness-of-fibers}Well roundedness of fibered
families }

In the previous section we have developed a machinery to establish
whether a fixed set $\Gset\subset G$ is LWR (Proposition \ref{prop: BCS is well rounded}),
and whether a family of the form $\left\{ P_{T}^{\prime}Q_{T}^{\prime}\right\} $
is LWR (Corollary \ref{cor: roundo to a produt grp}), where $P_{T}^{\prime}\subset P$,
$Q_{T}^{\prime}\subset Q$ and $G=PQ$ is a decomposition of $G$
into subgroups $P$ and $Q$. The latter applies for any number of
components in the decomposition, but for brevity we wrote it here
with two components only. In this section we will extend our machinery
to handle families of sets with the more complicated structure of
a fiber product, namely sets of the form $\cup_{z\in P_{T}^{\prime}}Q_{z}^{\prime}$,
where again $P_{T}^{\prime}\subset P$, $Q_{z}^{\prime}\subset Q$
and $G=PQ$. The tool of roundomorphisms allows us to reduce to well
roundedness in $P$ and in $Q$ separately, namely to work in $P\times Q$,
which is what we will do. 

We start by formulating a regularity condition on the fibers in $Q$. 
\begin{defn}
\label{def: BLC}Let $P$ and $H$ be Lie groups and $\nbhd{\e}P$
and $\nbhd{\e}H$ families of coordinate balls. Let $\symset$ be
a subset of $P$, and consider the family $\fam{\mathcal{\fdomN}}{\symset}=\left\{ \domN_{z}\right\} _{z\in\symset}$,
where $\domN_{z}\subseteq H$. We say that the family $\fam{\mathcal{\fdomN}}{\symset}$
is \emph{bounded Lipschitz continuous }(or \textbf{BLC}) w.r.t $\nbhd{\e}P$
and $\nbhd{\e}H$ and with parameters $\left(C_{\fdomN},\volmin,B\right)$,
where $C_{\fdomN},\volmin$ are positive real numbers and $B$ is
a bounded subset of $H$, if for every $0<\e<C_{\fdomN}^{-1}$ the
following hold:
\begin{enumerate}
\item Every $D_{z}$ is LWR with parameters $\left(C_{\fdomN},1/C_{\fdomN}\right)$. 
\item If $z^{\prime}\subseteq\nbhd{\e}Pz\nbhd{\e}P$ for $z,z^{\prime}\in\symset$,
then $\mathcal{D}_{z}^{\left(-C_{\fdomN}\epsilon\right)}\subseteq\domN_{z^{\prime}}\subseteq\mathcal{D}_{z}^{\left(+C_{\fdomN}\epsilon\right)}$.
\item The volume of $\domN_{z}$ (w.r.t. a Haar measure of $H$) is bounded
uniformly from below by a positive constant \textbf{$\volmin$}. 
\item $\domN_{z}\subseteq\ball{}$ for some bounded set $B$ and every $z\in\symset$. 
\end{enumerate}
For convenience, we will always assume WLOG that $C_{\fdomN}\geq1$. 
\end{defn}

The following proposition and corollary are concerned with certain
manipulations that can be performed on fibered sets, while maintaining
the BLC property of the fibers. These manipulations include pulling
back the fibers by a locally-Lipschitz map, and enlarging the basis
set by taking a product with another set. 
\begin{prop}
\label{prop: same fiber different space}Let $P,P_{0}$ be Lie groups
and suppose that $r:P_{0}\to P$ is an $f$-locally Lipschitz map
(Definition \ref{def: roundomorphism}). Let $\symset\subseteq P$
and $\mathcal{\symset}_{0}:=r^{-1}\left(\symset\right)\subseteq P_{0}$.
If the family 
\[
\fam{\fdomN}{\symset}=\left\{ \domN_{z}\right\} _{z\in\symset}
\]
is BLC with parameters $\left(C,\volmin,B\right)$, then the family
\[
\fam{\fdomN}{\mathcal{E}_{0}}=\left\{ \domN_{r\left(z_{0}\right)}\right\} _{z_{0}\in\mathcal{E}_{0}}
\]
is BLC with parameters $\left(FC,\volmin,B\right)$, where $F=\sup_{g\in r^{-1}\left(\symset\right)}f(g)<\infty$. 
\end{prop}

\begin{proof}
Since $\fam{\fdomN}{\mathcal{\symset}_{0}}\subset\fam{\fdomN}{\symset}$,
properties 1, 3, and 4 of BLC hold automatically in $\fam{\fdomN}{\mathcal{\symset}_{0}}$,
and it is only left to verify the second property. Indeed, if $z_{0}^{\prime}\in\nbhd{\e}{P_{0}}z_{0}\,\nbhd{\e}{P_{0}}$,
then by local Lipschitzity and definition of $F$, $r\left(z_{0}^{\prime}\right)\in\nbhd{F\e}Pr\left(z_{0}\right)\nbhd{F\e}P$.
Since $\fam{\fdomN}{\symset}$ is BLC then for $\epsilon\leq\frac{1}{FC}$
we obtain
\[
\domN_{r\left(z_{0}^{\prime}\right)}^{\left(-C\epsilon\right)}\subseteq\domN_{r\left(z_{0}\right)}\subseteq\domN_{r\left(z_{0}^{\prime}\right)}^{\left(+C\epsilon\right)}.
\]
\end{proof}
\begin{cor}
\label{cor: larger base space}Let $P\times Q$ be a product of Lie
groups and let $\symset\subseteq P$, $\mathcal{\mathcal{\symset^{\prime}}}=\symset\times Q$.
If $\fam{\fdomN}{\symset}=\left\{ \domN_{z}\right\} _{z\in\symset}$
is BLC w.r.t. $\nbhd{\e}P$, then 
\[
\fam{\fdomN}{\mathcal{E^{\prime}}}=\left\{ \domN_{\left(z,q\right)}\right\} _{\left(z,q\right)\in\mathcal{E^{\prime}}}\text{ such that \ensuremath{\domN_{\left(z,q\right)}=\domN_{z}\:\forall q\in Q}}
\]
is BLC with the same parameters and w.r.t. $\nbhd{\e}P\times\nbhd{\e}Q$. 
\end{cor}

\begin{rem}
Clearly we can replace the group $Q$ in the definition of $\symset^{\prime}$
with any subset $\Gset\subseteq Q$.
\end{rem}

\begin{proof}
This follows from Proposition \ref{prop: same fiber different space}
using the projection map
\[
r:P\times Q\to P
\]
which is an $f$-local Lipschitz map with $f\equiv1$.
\end{proof}
We now turn to the concluding result of this section:
\begin{prop}
\label{prop: Fibered is LWR in direct product}Let $\left\{ \symset_{T}\right\} _{T>0}$
be an increasing family inside a Lie group $P$, and $\symset:=\cup_{T>0}\symset_{T}$.
Let $\fam{\fdomN}{\symset}=\left\{ \domN_{z}\right\} _{z\in\symset}$
where $\domN_{z}\subset H$, and consider the family 
\[
\Gset_{T}=\bigcup_{z\in\symset_{T}}z\times\domN_{z}\subseteq P\times H.
\]
If $\left\{ \symset_{T}\right\} _{T>0}$ is LWR with parameters $\left(T_{0},C_{\symset}\right)$,
and $\mathcal{\fdomN}_{\symset}$ is BLC w.r.t. a family $\left\{ \nbhd{\e}P,\nbhd{\e}H\right\} _{\e>0}$
of coordinate balls and with parameters $\left(C_{\mathcal{\fdomN}},\volmin,B\right)$,
then $\Gset_{T}$ is LWR w.r.t the coordinate balls $\nbhd{\e}P\times\nbhd{\e}H\subset P\times H$
and with parameters $\left(T_{0},C_{\Gset}\right)$ where
\[
C_{\Gset}\prec C_{\fdomN}c\left(1+C_{\mathcal{\fdomN}}\right)+\frac{\volmax}{\volmin}C_{\symset},
\]
 $\volmax=\mu_{H}\left(\ball{}\right)$ and $c\geq1$ is a constant
such that for $0<\epsilon,\delta<\frac{1}{c}$ one has that $\nbhd{\e}H\nbhd{\delta}H\subseteq\nbhd{c\left(\e+\delta\right)}H$
(see Lemma \ref{lem: Connectivity + additivity of coord. balls}). 
\end{prop}

\begin{proof}
\textbf{Step 1: estimation of }$\Gset_{T}^{\left(+\e\right)}$\textbf{.}
We claim that for $\e<\frac{1}{cC_{\mathcal{\fdomN}}}$ (so $\e<1,\left(cC_{\mathcal{\fdomN}}\right)^{-1}$),
\[
\Gset_{T}^{\left(+\e\right)}\subseteq\left\{ \bigcup_{z\in\symset_{T}}\left(z\times\domN_{z}^{\left(+c\left(1+C_{\mathcal{\fdomN}}\right)\e\right)}\right)\right\} \bigcup\left\{ \minset\times B^{\left(+1\right)}\right\} =:Y^{+},
\]
where 
\[
\minset:=\nbhd{\e}P\symset_{T}\nbhd{\e}P\setminus\symset_{T}.
\]
We shall first bound the affect of $\nbhd{\e}P$ perturbations. For
that recall that $\domN_{z}\subseteq\ball{}$ for all $z\in\symset$.
As a result, for $u,v\in\nbhd{\e}P$ we have
\[
\left(v,e_{H}\right)\Gset_{T}\left(u,e_{H}\right)=\bigcup_{z\in\symset_{T}}\left(vzu\times\domN_{z}\right)\subseteq\left(\bigcup_{z\in\symset_{T}\cap v\symset_{T}u}\left(z\times\domN_{v^{-1}zu^{-1}}\right)\right)\bigcup\left(\minset\times\ball{}\right).
\]
By the second property of BLC, for $\e<\frac{1}{C_{\mathcal{\fdomN}}}$,
this is contained in 
\[
\bigcup_{z\in\symset_{T}}\left(z\times\domN_{z}^{\left(+C_{\mathcal{\fdomN}}\e\right)}\right)\bigcup\left(\minset\times\ball{}\right).
\]
We will now address the $\nbhd{\e}H$ perturbations. To this end,
note that by the first property of BLC, for $\e<\frac{1}{cC_{\mathcal{\fdomN}}}$
\[
\left(\domN_{z}^{\left(+C_{\mathcal{\fdomN}}\e\right)}\right)^{\left(+\epsilon\right)}\subseteq\domN_{z}^{\left(+c\left(1+C_{\mathcal{\fdomN}}\right)\e\right)}.
\]
Combining $\nbhd{\e}P$ and $\nbhd{\e}H$ perturbations together we
obtain,
\[
\Gset_{T}^{\left(+\e\right)}=\nbhd{\e}{}\Gset_{T}\nbhd{\e}{}\subseteq\bigcup_{z\in\symset_{T}}\left(z\times\domN_{z}^{\left(+c\left(1+C_{\mathcal{\fdomN}}\right)\e\right)}\right)\bigcup\left(\minset\times B^{\left(+1\right)}\right)=Y^{+}
\]
(where we have used $\e<1$). 

\textbf{Step 2: estimation of $\Gset_{T}^{\left(-\e\right)}$.} We
claim that for $\e<\frac{1}{cC_{\mathcal{\fdomN}}},$
\[
\Gset_{T}^{\left(-\e\right)}\supseteq\bigcup_{z\in}\left(z\times\domN_{z}^{\left(-c\left(C_{\mathcal{\fdomN}}+1\right)\e\right)}\right)=:Y^{-}.
\]
First notice that if $0<a,\frac{b}{c}-a<\frac{1}{c}$, then $\left(\domN^{\left(-b\right)}\right)^{\left(+a\right)}\subseteq D^{-\frac{b}{c}+a}$,
since 
\[
\nbhd{\frac{b}{c}-a}H\nbhd aH\left(\domN^{\left(-b\right)}\right)\nbhd aH\nbhd{\frac{b}{c}-a}H\subseteq\nbhd bH\left(\domN^{\left(-b\right)}\right)\nbhd bH\subseteq\domN.
\]
Hence, for $\epsilon<\frac{1}{cC_{\mathcal{\fdomN}}}$
\[
\nbhd{\e}HY^{-}\nbhd{\e}H=\bigcup_{z\in\symset_{T}^{\left(-\e\right)}}\left(z\times\left(\domN_{z}^{\left(-\left(C_{\mathcal{\fdomN}}+1\right)c\epsilon\right)}\right)^{\left(+\e\right)}\right)\subseteq\bigcup_{z\in\symset_{T}^{\left(-\e\right)}}\left(z\times\domN_{z}^{\left(-C_{\mathcal{\fdomN}}\epsilon\right)}\right)
\]
for $u,v\in\nbhd{\e}P$ we have
\[
\left(v,e_{H}\right)\nbhd{\e}HY^{-}\nbhd{\e}H\left(u,e_{H}\right)\subseteq\bigcup_{z\in\symset_{T}^{\left(-\e\right)}}\left(vzu\times\domN_{z}^{\left(-C_{\mathcal{\fdomN}}\e\right)}\right)\subseteq\bigcup_{v^{-1}zu^{-1},z\in\symset_{T}}\left(z\times\domN_{v^{-1}zu^{-1}}^{\left(-C_{\mathcal{\fdomN}}\e\right)}\right).
\]
By the second property of BLC, for $\e<\frac{1}{C_{\mathcal{\fdomN}}}$
this is contained in 
\[
\bigcup_{z\in\symset_{T}}\left(z\times\domN_{z}\right).
\]
All in all, we obtain that $\nbhd{\e}{}Y^{-}\nbhd{\e}{}\subseteq\Gset_{T}$,
proving the claim.

\textbf{Step 3: estimation of $\mu\left(\Gset_{T}^{\left(+\e\right)}\right)/\mu\left(\Gset_{T}^{\left(-\e\right)}\right)$.}
Let $C=C_{\fdomN}\left(c\left(1+C_{\mathcal{\fdomN}}\right)\right)$.
Notice that for $\e<\frac{1}{C}$ 
\begin{align*}
\mu_{G}\left(Y^{+}\right) & =\left(1+C\e\right)\mu_{G}\left(\Gset_{T}\right)+\mu_{P}\left(\minset\right)\mu_{H}\brac{\ball{}^{\left(+1\right)}}\\
 & \leq\left(1+C\e\right)\mu_{G}\left(\Gset_{T}\right)+\mu_{P}\left(\minset\right)\volmax
\end{align*}
and that 
\begin{align*}
\mu_{G}\left(Y^{-}\right) & =\frac{1}{1+C\e}\:\mu_{G}\left(\bigcup_{z\in\symset_{T}^{\left(-\e\right)}}\left(z\times\domN_{z}\right)\right).
\end{align*}
Combining what we have shown in the previous steps with estimations
for $\mu_{G}\left(Y^{+}\right)$ and $\mu_{G}\left(Y^{-}\right)$,
we obtain that for $\e<\frac{1}{C}$:
\begin{align*}
 & \frac{\mu_{G}\left(\Gset_{T}^{\left(+\e\right)}\right)}{\mu_{G}\left(\Gset_{T}^{\left(-\e\right)}\right)}\leq\frac{\mu_{G}\left(Y^{+}\right)}{\mu_{G}\left(Y^{-}\right)}\\
 & \leq\left(1+C\e\right)^{2}\frac{\mu_{G}\left(\Gset_{T}\right)}{\mu_{G}\left(\bigcup_{z\in\symset_{T}^{\left(-\e\right)}}\left(z\times\domN_{z}\right)\right)}+\volmax\left(1+C\e\right)\cdot\frac{\mu_{P}\left(\minset\right)}{\mu_{G}\left(\bigcup_{z\in\symset_{T}^{\left(-\e\right)}}\left(z\times\domN_{z}\right)\right)}
\end{align*}
where:
\begin{itemize}
\item for $\e<1$,
\[
\left(1+C\e\right)^{2}\leq1+3C\e
\]
 
\item for $\e<\frac{1}{C}$,
\[
\volmax\left(1+C\e\right)\leq2\volmax
\]
\item for $\e<C_{\symset}^{-1}$ and $T>T_{0}$
\[
\frac{\mu_{G}\left(\Gset_{T}\right)}{\mu_{G}\left(\bigcup_{z\in\symset_{T}^{\left(-\e\right)}}\left(z\times\domN_{z}\right)\right)}=1+\frac{\mu_{G}\left(\Gset_{T}\right)-\mu_{G}\left(\bigcup_{z\in\symset_{T}^{\left(-\e\right)}}\left(z\times\domN_{z}\right)\right)}{\mu_{G}\left(\bigcup_{z\in\symset_{T}^{\left(-\e\right)}}\left(z\times\domN_{z}\right)\right)}=1+\frac{\mu_{G}\left(\bigcup_{z\in\symset_{T}\setminus\symset_{T}^{\left(-\e\right)}}\left(z\times\domN_{z}\right)\right)}{\mu_{G}\left(\bigcup_{z\in\symset_{T}^{\left(-\e\right)}}\left(z\times\domN_{z}\right)\right)}
\]
\[
\leq1+\frac{\mu_{P}\left(\symset_{T}\setminus\symset_{T}^{\left(-\e\right)}\right)\volmax}{\mu_{P}\left(\symset_{T}^{\left(-\e\right)}\right)\volmin}\leq1+\frac{\mu_{P}\left(\symset_{T}^{\left(+\e\right)}\setminus\symset_{T}^{\left(-\e\right)}\right)}{\mu_{P}\left(\symset_{T}^{\left(-\e\right)}\right)}\cdot\frac{\volmax}{\volmin}\leq1+\frac{\volmax}{\volmin}C_{\symset}\e
\]
\item and for $\e<C_{\symset}^{-1}$ and $T>T_{0}$, 
\[
\frac{\mu_{P}\left(\minset\right)}{\mu_{G}\left(\bigcup_{z\in\symset_{T}^{\left(-\e\right)}}\left(z\times\domN_{z}\right)\right)}=\frac{\mu_{P}\left(\nbhd{\e}P\symset_{T}\nbhd{\e}P\setminus\symset_{T}\right)}{\mu_{G}\left(\bigcup_{z\in\symset_{T}^{\left(-\e\right)}}\left(z\times\domN_{z}\right)\right)}\leq\frac{\mu_{P}\left(\symset_{T}^{\left(+\e\right)}\setminus\symset_{T}^{\left(-\e\right)}\right)}{\mu_{P}\left(\symset_{T}^{\left(-\e\right)}\right)\cdot\volmin}\leq\frac{C_{\symset}}{\volmin}\e.
\]
\end{itemize}
All in all, for $\e<\frac{1}{C+C_{\symset}}$ (so that $\e\leq C^{-1},C_{\symset}^{-1}$)
and for $T>T_{0}$:
\begin{align*}
\frac{\mu_{G}\left(Y^{+}\right)}{\mu_{G}\left(Y^{-}\right)} & \leq\left(1+3C\e\right)\cdot\left(1+\frac{\volmax}{\volmin}C_{\symset}\e\right)+2\volmax\cdot\frac{C_{\symset}}{\volmin}\e\\
 & \leq1+\left(6\frac{\volmax}{\volmin}C_{\symset}+3C\right)\e.
\end{align*}
In order to have that LWR holds for $\e<C_{\Gset}^{-1}$, we let
$C_{\Gset}=6\frac{\volmax}{\volmin}C_{\symset}+3C$. 
\end{proof}
\begin{prop}
\label{prop: BLC for R^n}Let $H=\mathbb{R}^{n}$ and \textup{$\nbhd{\epsilon}H=B_{\epsilon}$
be a radius $\epsilon$ euclidean ball. Suppose that $\fam{\fdomN}{\symset}$
satisfies conditions 3 and 4 of the definition of BLC and instead
of condition 1 and 2 it satisfies that for $\epsilon<C^{-1}$ (here
$C\geq1$):}\renewcommand{\labelenumi}{(\roman{enumi})} 
\begin{enumerate}
\item $\domN_{z}+\ball{\e}\subseteq\left(1+C\e\right)\domN_{z}$
\item If $z^{\prime}\subseteq\nbhd{\e}Pz\nbhd{\e}P$ for $z,z^{\prime}\in\symset$,
then $\domN_{z^{\prime}}\subseteq\left(1+C\e\right)\domN_{z}$.
\end{enumerate}
Then $\fam{\fdomN}{\symset}$ is BLC with parameters $\left(16^{n+1}RC,\volmin,B\right)$,
where $R<\infty$ is the radius of $B$ from property 4 of BLC).
\end{prop}

\begin{proof}
We start with checking the first property of BLC. For this we first
show that for all $z\in\mathcal{E}$ and $\epsilon<\frac{1}{4C}$
\begin{equation}
\frac{1}{1+8C\epsilon}\mathcal{D}_{z}\subseteq\mathcal{D}_{z}^{\left(-\epsilon\right)}.\label{eq: equation}
\end{equation}
Indeed, since $\epsilon<C^{-1}$ then by (i) $\domN_{z}+\ball{\e}\subseteq\left(1+C\e\right)\domN_{z}$,
and so
\[
\frac{1}{1+2C\epsilon}\mathcal{D}_{z}+B_{\frac{\epsilon}{2}}\subseteq\frac{1}{1+2C\epsilon}\mathcal{D}_{z}+B_{\frac{\epsilon}{1+C\epsilon}}\subseteq\frac{1}{1+C\epsilon}\left(\mathcal{D}_{z}+B_{\epsilon}\right)\subseteq\domN_{z}.
\]
As a result, 
\[
\frac{1}{1+8C\epsilon}\mathcal{D}_{z}+\ball{2\e}\subseteq\mathcal{D}_{z}
\]
which implies \ref{eq: equation}. Now, for $\epsilon<\frac{1}{16C}$
we have that
\[
\frac{\mu_{H}\left(\mathcal{D}_{z}^{\left(+\epsilon\right)}\right)}{\mu_{H}\left(\mathcal{D}_{z}^{\left(-\epsilon\right)}\right)}\leq\frac{\mu_{H}\left(\left(1+2C\e\right)\domN_{z}\right)}{\mu_{H}\left(\frac{1}{1+8C\epsilon}\mathcal{D}_{z}\right)}=\left(\frac{1+2C\e}{\frac{1}{1+8C\epsilon}}\right)^{n}<\left(1+11C\epsilon\right)^{n}<1+16^{n+1}C\epsilon.
\]
So $\left\{ \mathcal{D}_{z}\right\} $ is LWR with Lipschitz constant
$16^{n+1}C$. 

For the second property of BLC we first show that for all $\epsilon>0$
\[
\mathcal{D}_{z}^{\left(-\epsilon\right)}\subseteq\frac{1}{1+\frac{\epsilon}{R}}\mathcal{D}_{z}\:;\quad\left(1+C\e\right)\mathcal{D}_{z}\subseteq\mathcal{D}_{z}^{\left(+CR\epsilon\right)}.
\]
Indeed, If $x\in\mathcal{D}_{z}^{\left(-\epsilon\right)}$, then $x+\text{\ensuremath{\frac{x}{R}}}\epsilon\in\mathcal{D}_{z}^{\left(-\epsilon\right)}+\ball{\e}\subseteq\mathcal{D}_{z}$.
Hence, $x\in\frac{1}{1+\frac{\epsilon}{R}}\mathcal{D}_{z}$ and so
$\mathcal{D}_{z}^{\left(-\epsilon\right)}\subseteq\frac{1}{1+\frac{\epsilon}{R}}\mathcal{D}_{z}$.
Next, if $y\in\left(1+C\e\right)\mathcal{D}_{z}$, we can write $y=x+C\epsilon x$
for $x\in\mathcal{D}_{z}$ and so $y\in\mathcal{D}_{z}+\ball{RC\e}\in\mathcal{D}_{z}^{\left(+CR\epsilon\right)}$.
Now for $\epsilon<\frac{1}{C}$ let $z\in\symset$ and $z^{\prime}\subseteq\nbhd{\e}Pz\nbhd{\e}P\cap\symset$
, and then by (ii),
\[
\domN_{z^{\prime}}\subseteq\left(1+C\e\right)\domN_{z}\subseteq\mathcal{D}_{z}^{\left(+CR\epsilon\right)}
\]
 and 
\[
\mathcal{D}_{z}^{\left(-RC\epsilon\right)}\subseteq\frac{1}{1+C\e}\domN_{z}\subseteq\mathcal{D}_{z^{\prime}}.
\]
So the second property of BLC holds with the constant $RC$, and all
in all, both first and second properties are satisfied with $C_{\fdomN}=16^{n+1}CR$.
\end{proof}

\begin{cor}
\label{cor: BCS is LCHC}Assume $\domN\subseteq\RR^{n}$ is bounded,
convex and has a non-empty interior, then $\mathcal{D}$ is LWR.
\end{cor}

\begin{proof}
It is clearly enough to check the case where the origin is an internal
point. We will show that the constant family $\fam{\mathcal{\fdomN}}{\symset}=\left\{ \domN_{z}\right\} _{z\in\symset}$
with $\domN_{z}=\domN$ for every $z$, is BLC as a set in $\RR^{n}$
using Proposition \ref{prop: BLC for R^n}. The second property of
BLC is trivial since $\domN$ is constant, and the third and fourth
properties hold since $\domN$ is bounded and of positive measure.
It remains to show that $\domN$ satisfies the first property of BLC.
Let $\a>0$ be such that $\domN$ contains a ball of radius $>\a$
around the origin; we show that $\domN+\ball{\e}\subseteq\left(1+\a^{-1}\e\right)\domN$.
Indeed, let $x\in\domN$ and $v\in\RR^{n}$ such that $\norm v=1$.
Then 
\begin{align*}
x+\e v & =\left(x+\e v\right)\exup{\left(1+\frac{\e}{\a}\right)\frac{1}{1+\frac{\e}{\a}}}{_{=1}}=\left(1+\frac{\e}{\a}\right)\cdot\frac{x+\e v}{1+\frac{\e}{\a}}\\
= & \left(1+\frac{\e}{\a}\right)\exd{\left(\frac{1}{1+\frac{\e}{\a}}\cdot x+\frac{\frac{\e}{\a}}{1+\frac{\e}{\a}}\cdot\a v\right)}{\left(\star\right)}
\end{align*}
where $\left(\star\right)$ lies in $\domN$, as a convex combination
of the two points $x$, $\a v$ in $\domN$. 
\end{proof}

\section{\label{sec: spread models}Relation between fundamental domains and
quotient spaces 
\global\long\def\Hbcs{B_{H}}%
}

As mentioned in the Introduction, this note is meant to support our
work on equidistribution in various lattice spaces. All of these spaces
are of the form $\mathcal{M}/H$, where $\mathcal{M}$ is a manifold
and $H$ is a Lie group acting on it. We are interested in Boundary
Controllable Sets (and Bounded Lipschitz Continuous families of sets)
in these spaces, which is a differential property; as such, it is
easier to check it in concrete manifolds, than in abstract spaces.
This bring up the need in finding a set of representatives inside
$\manifold$ for the action of $H$, such that one can move the question
of verifying the BCS property from the space $\manifold/H$ to this
subset of $\manifold$. In the case where $H$ is discrete, one can
think of this desired set of representatives as a ``nice'' fundamental
domain in $\manifold$ (see examples below). The precise definition
is the following:
\begin{defn}
\label{def: Spread Model}Let $H$ be a Lie group acting smoothly,
freely and properly on a manifold $\manifold$, and let $\pi:\manifold\to\manifold/H$
denote the associated quotient map. A full set of representatives
$F\subset\manifold$ for $\manifold/H$ is called a \emph{spread model}
for the quotient space $\manifold/H$ if the following conditions
are met: 
\begin{enumerate}
\item $F$ is contained in a finite union of embedded submanifolds $\cup_{\alpha}V_{\alpha}$
of $\manifold$ such that the natural projection $\pi_{\alpha}^{0}:V_{\alpha}\to\mathcal{M}/H_{0}$
(here $H_{0}$ denotes the connected component of $H$) is an open
diffeomorphism onto its image;
\item for each $\alpha$, there exists an open set (w.r.t.\ $V_{\a}$)
$F_{\a}\subseteq F\cap V_{\a}$ such that $\overline{F}\cap V_{\alpha}\subset\overline{F_{\alpha}}$;
\item $F_{\alpha}$ is BCS w.r.t.\ $V_{\alpha}$.  and
\item the quotient map restricted to $\overline{F}$ is proper, namely it
pulls back compact sets to compact sets. 
\end{enumerate}
We will denote $F\smeq\manifold/H$.
\end{defn}

\begin{rem}
One can extend definition also to the case where $H$ acts \textbf{almost
freely} (i.e. the point stabilizer subgroups are finite) on $\mathcal{M}$
by considering the open submanifold $\mathcal{M}_{free}$ on which
the action is free. In that case we add the extra condition that 
$F\subseteq\mathcal{M}$ is a spread model if $F\cap\mathcal{M}_{free}$
is a spread model for the action of $H$ on $\mathcal{M}_{free}$.
\end{rem}

\begin{rem}
\label{rem:sm for discrete groups}If $H=\Gamma$ is a discrete Lie
group, then the conditions in Definition \ref{def: Spread Model}
are satisfied when: $\Gamma$ acts properly and almost freely  on
$\manifold$; $F\subseteq\overline{\interior F}$; the boundary of
$F$ is contained in a finite union of lower dimensional submanifolds
of $\manifold$ and the quotient map restricted to $\overline{F}$
is proper. This is indeed the case since the quotient map $\pi:\manifold_{free}\to\manifold_{free}/\Gamma$
restricted to $\interior F$ is a  diffeomorphism.
\end{rem}

\begin{example}
Let us mention some examples (an explanation follows). 
\begin{enumerate}
\item If $\Lambda$ is a lattice in $\mathbb{R}^{n}$, then any fundamental
domain of $\lat$ which is a convex polygon with a finite number of
edges, is a spread model for $\RR^{n}/\lat$. In particular, its fundamental
parallelepiped and its Dirichlet domain are spread models. 
\item If $\Gamma$ is a discrete group acting properly discontinuously and
freely on $\mathbb{H}^{n}$ by isometries, then any locally finite
fundamental polygon $F$ is a spread model. In particular, the (generic)
Dirchilet domain of a Fuchsian  group is a spread model.
\item Suppose that $G=H\cdot P$, is a product of two closed subgroups having
trivial intersection. Then $P$ is a spread model for the action of
$H$. In particular, the group of upper triangular matrices in $\sl n\left(\mathbb{R}\right)$
is a spread model for the symmetric space $\so n\left(\mathbb{R}\right)\backslash\sl n\left(\mathbb{R}\right)$,
and a minimal parabolic group inside a non-compact simple algebraic
rank one Lie group is a spread model for the associated hyperbolic
space. E.g.\, the upper triangular matrices in $\sl 2\left(\RR\right)$
(resp.\ $\sl 2\left(\CC\right)$) form a spread model for the real
hyperbolic plane (resp.\ 3-space). 
\end{enumerate}
Indeed, we use Remark \ref{rem:sm for discrete groups} to justify
the first two examples. All the conditions are easily seen to be satisfied
except that $\pi|_{\overline{F}}$ is proper in the second example.
We check it here: suppose that $a_{n}\to\infty$ in $\overline{F}$,
but modulo $\Gamma$, the set $\left\{ \Gamma a_{n}\right\} $ is
bounded. As a result, we may assume that $\gamma_{n}.a_{n}\to a$
for some $\gamma_{n}\in\Gamma$. Let $K$ be a compact neighborhood
of $a$. We have that $K\cap\gamma_{n}\overline{F}\neq\emptyset$
for all $n$. Since $F$ is locally finite, we may assume that $\gamma_{n}=\gamma$
for all $n$. This is however an absurd, since we must have that $\gamma.a_{n}\to\infty$.

To check example 3 we use the original definition and notice that
we may choose $V_{\alpha}=F_{\alpha}=P$ (only a single $\alpha$)
and since $P=\overline{P}$ is diffeomorphic via $\pi$ to $H\backslash G$
and the boundary of $F_{\alpha}$ w.r.t. $V_{\alpha}$ is trivial,
we are done.
\end{example}

\subsection{BCS's in space and its spread model correspond}

We start by claiming that a BCS in the spread model projects modulo
$H$ to a BCS in the space $\manifold/H$.
\begin{prop}
\label{prop: BCS from manifold to space}Suppose that $\mathcal{M}$
is a manifold and $\Gamma$ is a discrete group acting on $\mathcal{M}$
freely, properly and smoothly. Let $F\subset\mathcal{M}$ be a spread
model for the action of $\Gamma$. If $B\subset F$ is a BCS (w.r.t.\
$\mathcal{M}$) then so is its projection $\pi\left(B\right)$ to
$\mathcal{M}/\Gamma$.
\end{prop}

The proof requires a lemma. 

\begin{lem}
\label{lem:boundary calculations}Let $X$ be a topological space
together with three subsets $A\subseteq B$ and $C$.
\begin{enumerate}
\item $\del_{B}A\subseteq\del A.$ If furthermore there is am open subset
$W$ of $X$, such that $\overline{A}\subseteq W\subseteq B$, then
$\del_{B}A=\del A$.
\item $\left(\del C\right)\cap A\subseteq\del_{B}\left(C\cap A\right)\cup\del B$.
If furthermore $B$ is open, then $\left(\del C\right)\cap A\subseteq\del_{B}\left(C\cap A\right)$.
\end{enumerate}
\end{lem}

\begin{proof}
1) Let $b\in\del_{B}A$ and let $U$ be a neighborhood of $b$. The
set $U\cap B$ is a neighborhood of $b$ w.r.t.\ $B$, hence it contains
a point of $A$ and a point of $A^{c}$. As a result, $b\in\del A$.

We need to show that $\del A\subseteq\del_{B}A$. Let $b\in\del A$
and let $U^{\prime}$ be a neighborhood of $b$ w.r.t.\ $B$. Hence
there is a neighborhood $U$ of $b$ such that $U^{\prime}=U\cap B$.
Since $b\in\overline{A}\subseteq W$ and $W$ is open, $U\cap W$
is a neighborhood of $b$ and so $U\cap W=U\cap W\cap B$ is also
a neighborhood of $b$ w.r.t.\ $B$. Hence, $U\cap W$ contains a
point of $A$ and a point of $A^{c}\cap B$. We conclude that indeed
$b\in\del_{B}A$.

2) Assume $x\in\left(\del C\right)\cap A$. If every neighborhood
of $x$ intersects $B^{c}$, we get that $x\in\del B$ (as $x\in A\subseteq B$).
Otherwise, there is a neighborhood $V$ of $x$ such that $V\subseteq B$.
Let $W'$ be a neighborhood of $x$ w.r.t.\ $B$, so that $W^{\prime}=W\cap B$
where $W$ is a neighborhood of $x$. Since $W\cap V$ is a neighborhood
of $x$ in $X$, we know that it intersects both $C$ and $C^{c}$.
Since $W\cap V=W^{\prime}\cap V$ we get that it is also a neighborhood
of $x$ w.r.t.\ $B$. Hence, $x\in\del_{B}\left(C\cap A\right)$. 

In the case when $B$ is open, since $x\in A\subset B$, $B$ is a
neighborhood of $x$ which does not intersect $B^{c}$.
\end{proof}
\begin{proof}[Proof of Proposition \ref{prop: BCS from manifold to space}]
We split $B$ into two parts: $B\cap\del F$ and $B_{\text{int}}:=B\cap F^{\circ}$.
Since, by part (2) of Lemma \ref{lem:boundary calculations}, $\del\pi\left(B\right)\subseteq\pi\left(\del F\right)\cup\del_{\pi\left(F^{\circ}\right)}\pi\left(B_{\text{int}}\right)$,
it is enough to show that $\pi\left(B_{\text{int}}\right)$ is BCS
inside $\pi\left(F^{\circ}\right)$ and that locally $\pi\left(\del F\right)$
is contained in a finite union of codimension $\geq1$ submanifolds.

The first part is clear since $\pi|_{F^{\circ}}$ is an open diffeomorphism
onto its image $\pi\left(F^{\circ}\right)$. 

The second part follows from $\del F$, by assumption, being locally
contained in a finite union of codimension $\geq1$ submanifolds together
with $\pi$ being a local diffeomorphism.
\end{proof}
The content of the following result is the converse of Proposition
\ref{prop: BCS from manifold to space}. In the case where $H$ is
discrete, this merely means that the lift of a BCS in $\manifold/H$
is a BCS in $F\subset\manifold$. When $H$ is not discrete, i.e.\
$\dim\left(\manifold/H\right)<\dim\left(\manifold\right)$, this statement
has no actual content since $F$ is its own boundary; a more delicate
formulation is therefore in order:
\begin{prop}
\label{prop: BCS from space to manifold}Assume that $F$ is a spread
model for $\manifold/H$, where $H$ acts almost freely and properly
on $\mathcal{M}$ and $H_{0}$ acts freely on $\mathcal{M}$. Then
if $B\subseteq\manifold/H$ and $\Hbcs\subseteq H$ are BCS (resp.\.
bounded), then so does
\[
B_{F}\cdot\Hbcs\subseteq\manifold,
\]
 where $B_{F}=\pi\vert_{F}^{-1}\left(B\right)$.
\end{prop}

The proof requires a lemma. The condition regarding $V_{\a}$ appearing
in the first part of Definition \ref{def: Spread Model} can be restated
as in the following, which is probably already known:
\begin{lem}
\label{lem:FHIsM}Assume that $V_{\alpha}$ is a submanifold of $\mathcal{M}$
and $H$ acts freely and properly on $V_{\alpha}$. The natural projection
$\pi_{\alpha}^{0}:V_{\alpha}\to\mathcal{M}/H_{0}$ ($\pi_{\alpha}:F_{\alpha}\to\mathcal{M}/H$)
is an open diffeomorphism onto its image iff the map $\theta_{\alpha}:V_{\alpha}\times H_{0}\to\mathcal{M}$
($F_{\alpha}\times H\to\mathcal{M}$) given by $\theta_{\alpha}\left(x,h\right)=x\cdot h$
is an open diffemorphism onto its image.
\end{lem}

\begin{proof}
It is clearly enough to prove the Lemma for $H_{0}$. Assume that
$\theta_{\alpha}$ is an open diffemorphism onto its image. It is
sufficient to show that $\pi_{\alpha}^{0}$ is an injective (this
is clear) submersion and that $\dim V_{\alpha}=\dim\mathcal{M}/H_{0}$.
Consider the diagram:
\[
\xymatrix{V_{\alpha}\ar[r]^{\iota_{e}}\ar[rrd]^{\pi_{\alpha}^{0}} & V_{\alpha}\times H_{0}\ar[r]^{\theta_{\alpha}} & \mathcal{M}\ar[d]^{\pi^{0}}\\
 &  & \mathcal{M}/H_{0}
}
,
\]
where $\iota_{h}:V_{\alpha}\to V_{\alpha}\times H_{0}$ is given by
$\iota_{h}\left(v\right)=\left(v,h\right)$. Since $\pi^{0}\circ\theta_{\alpha}|_{\left\{ p\right\} \times H_{0}}$
is a constant map for every $p\in V_{\alpha}$, then $d\left(\iota_{e}\circ\pi^{0}\circ\theta_{\alpha}\right)$
and $d\left(\pi^{0}\circ\theta_{\alpha}\right)$ have the same image.
The maps $\theta_{\alpha}$ and $\pi^{0}$ are submersions, hence
so is $\pi_{\alpha}^{0}$. Finally, since $\dim V_{\alpha}+\dim H_{0}=\dim\mathcal{M}=\dim\mathcal{M}/H_{0}+\dim H_{0}$,
we get the desired dimension equality.

Assume now that $\pi_{\alpha}^{0}$ is an open diffemorphism onto
its image. It is sufficient to show that $\theta_{\alpha}$ is an
injective (which is again clear) immersion and that $\dim V_{\alpha}+\dim H_{0}=\dim\mathcal{M}$.
For every $h\in H_{0}$ consider the diagram:
\[
\xymatrix{V_{\alpha}\ar[r]^{\iota_{h}} & V_{\alpha}\times H_{0}\ar[r]^{\theta_{\alpha}} & \mathcal{M}\ar[r]^{\pi^{0}} & \mathcal{M}/H_{0}\ar[r]^{\left(\pi_{\alpha}^{0}\right)^{-1}} & V_{\alpha}}
.
\]
Since the composition of all of the maps in the diagram is the identity
map, we get that for any given point $\left(p,h\right)\in V_{\alpha}\times H_{0}$
and $\left(X,Y\right)\in T_{p}V_{\alpha}\times T_{h}H_{0}$, the equality
$d_{\left(p,h\right)}\theta_{\alpha}\left(X,Y\right)=0$ forces $X=0$.
It is well known that, under the assumption that $H_{0}$ is a Lie
group acting freely and properly on $\mathcal{M}$, $pH_{0}$ is a
closed submanifold of $\mathcal{M}$ and $d\theta_{\alpha}|_{\left\{ p\right\} \times H_{0}}$
is a diffeomorphism onto $pH_{0}$. As a result, $Y=0$ and so $\theta_{\alpha}$
is an immersion. Furthermore, if $xh=x^{\prime}h^{\prime}$ where
$x,x^{\prime}\in V_{\alpha}$ and $h,h^{\prime}\in H_{0}$, then $\pi_{\alpha}^{0}\left(x^{\prime}\right)=\pi_{\alpha}^{0}\left(x\right)$.
Since $\pi_{\alpha}^{0}$ is injective, we must have that $x=x^{\prime}$.
The action of $H_{0}$ is free, so $h=h^{\prime}$ and so $\theta_{\alpha}$
is injective. Finally, $\dim V_{\alpha}=\dim\mathcal{M}/H_{0}$ and
$\dim\mathcal{M}=\dim\mathcal{M}/H_{0}+\dim H_{0}$, so $\dim V_{\alpha}+\dim H_{0}=\dim\mathcal{M}$.
\end{proof}
\begin{proof}[Proof of Proposition \ref{prop: BCS from space to manifold}]
We first prove that if $B$ and $B_{H}$ are compact, then so is
$\overline{B_{F}\Hbcs}$. Since $B$ is compact and $\pi|_{\overline{F}}$
is proper, $\overline{B_{F}}$ is compact. As a result, it is enough
to show that $\overline{B_{F}\Hbcs}\subset\overline{B_{F}}\cdot\overline{\Hbcs}$:
indeed if $B_{F}\ni x_{n}$ and $B_{H}\ni h_{n}$ are such that $x_{n}h_{n}\to m$,
then by compactness we may pass to subsequences and assume also that
$x_{n}\to x$ and $h_{n}\to h$, so that $xh=m$. This proves the
claim. 

It remains to show that $\del\left(B_{F}\Hbcs\right)$ is locally
contained in a finite union of lower dimensional submanifolds of $\manifold$.
Let $\left\{ V_{\alpha}\right\} $ and $\left\{ F_{\alpha}\right\} $
as in the assumptions. Since $F_{\alpha}$ is open in $V_{\alpha}$,
we have that $\pi|_{F_{\alpha}}$ is a diffeomorphism onto its image
(which is an open submanifold of $\manifold/H$).

To proceed we assume first that the action of $G$ is free. Write
$W_{\alpha}=\pi\left(F_{\alpha}\right)$ and consider the map

\[
\tau_{\alpha}:W_{\alpha}\times H_{0}\to\exd{\left(\pi\right)^{-1}\left(W_{\alpha}\right)=F_{\alpha}H_{0}}{\subset\manifold}
\]
given by 
\[
\tau_{\alpha}\left(u,h\right)=\left(\left(\pi|_{F_{\alpha}}\right)^{-1}\left(u\right)\right)\cdot h.
\]
By Lemma \ref{lem:FHIsM} (notice that since $\pi_{\alpha}^{0}$ is
an open diffeomorphism, then $\pi_{\alpha}$ must also be such), this
is an open diffeomorphism which satisfies $\tau_{\alpha}\left(\pi\left(x\right),h\right)=x\cdot h$,
where $x\in F_{\alpha}$. By part 2 of Lemma \ref{lem:boundary calculations},
\begin{align*}
V_{\alpha}H\cap\del\left(B_{F}\Hbcs\right) & \subseteq\left(\del\left(B_{F}\Hbcs\right)\cap F_{\alpha}\Hbcs\right)\cup\del\left(\left(B_{F}\Hbcs\right)\cap\left(\overline{F_{\alpha}}\setminus F_{\alpha}\right)\Hbcs\right)\\
 & \subseteq\del_{F_{\alpha}H}\left(B_{F}\Hbcs\cap F_{\alpha}\Hbcs\right)\cup\del\left(F_{\alpha}\Hbcs\right)\cup\left(\del F_{\alpha}\cdot B_{H}\right)\\
 & \subseteq\del_{F_{\alpha}H}\left(\left(B_{F}\cap F_{\alpha}\right)\Hbcs\right)\cup\del\left(F_{\alpha}B_{H}\right).
\end{align*}
It is enough to show that the sets $\left(B_{F}\cap F_{\alpha}\right)\Hbcs$
and $F_{\alpha}B_{H}$ are BCS's. For the first one, we have
\[
\del_{F_{\alpha}H}\left(\left(B_{F}\cap F_{\alpha}\right)\Hbcs\right)=\tau_{\alpha}\left(\del_{W_{\alpha}\times H}\left(\left(F_{\alpha}\cap B\right)\times\Hbcs\right)\right)\subseteq\tau_{\alpha}\left(\del\left(\left(B\times\Hbcs\right)\cap\left(W_{\alpha}\times H\right)\right)\right).
\]
Since $B\times\Hbcs$ is BCS and $W_{\alpha}\times H$ is open in
$\manifold/H\times H$, then $\left(B\times\Hbcs\right)\cap\left(W_{\alpha}\times H\right)$
is BCS w.r.t.\ $W_{\alpha}\times H$. Since $\tau_{\alpha}$ is a
diffeomorphism, $\del\left(B_{F}\cap F_{\alpha}\right)\Hbcs$ is BCS
w.r.t.\ $F_{\alpha}H$. For the second set, write $B_{H}\subseteq\cup_{\beta}H_{0}h_{\beta}$
for some $\left\{ h_{\beta}\right\} $. By part 2 of Lemma \ref{lem:boundary calculations},
\[
\del\left(F_{\alpha}B_{H}\right)\cap F_{\alpha}H_{0}h_{\beta}\subseteq\del_{V_{\alpha}H_{0}b_{\beta}}\left(F_{\alpha}B_{H}\cap V_{\alpha}H_{0}h_{\beta}\right)=\left(\theta_{\alpha}^{-1}\circ\left(Id_{F_{\alpha}}\times R_{h_{\beta}}\right)\right)\left(\del\left(F_{\alpha}\times\left(B_{H}\cap H_{0}h_{\beta}\right)\right)\right),
\]
where $R_{\beta}:\manifold\to\manifold$, given by $R_{\beta}\left(x\right)=xh_{\beta}$,
is a diffeomorphism. The claim now follows using the third assumption,
$\Hbcs$ being a BCS and $H_{0}h_{\beta}$ being open in $H$.

We now turn to the general case where the $G$ action is almost proper.
In that case, by, \cite[Theorem 2.8.5]{Lie_Groups} and \cite[Theorem 4.3.5]{Sniatycki}\},
both $\mathcal{M}_{free}^{c}$ and $\pi\left(\mathcal{M}_{free}\right)^{c}$
are BCS's, since every point $x$ in them contains an open neighborhood
$U_{x}$, such that $U_{x}\cap\mathcal{M}_{free}^{c}$ ($U_{x}\cap\pi\left(\mathcal{M}_{free}\right)^{c}$)
is a finite union of codimension $\geq2$ submanifolds of $\mathcal{M}$.
Hence we get that $\left(B_{F}\cdot B_{H}\right)\cap\mathcal{M}_{free}$
and $\mathcal{M}_{free}^{c}$ are both BCS's. All in all, $B_{F}\cdot B_{H}$
is a BCS.
\end{proof}

\subsection{Measures on the space and its spread model correspond}

We proceed with a short discussion about measures. We will use the
the following:
\begin{thm}[\cite{Jus18}]
 Let $H$ be a unimodular Radon lcsc group and let $\mu$ be a $H$-invariant
Radon measure on an lcsc space $Y$. Assume that the $H$ action of
$H$ on $Y$ is strongly proper (i.e. the action is proper and the
quotient space $Y/H$ is lcsc). Then, for a Haar measure $\mu_{H}$
on $H$ there exists a unique Radon measure $\mu_{Y/H}$ on $Y/H$
such that for all $f\in C_{c}\left(Y\right)$, 
\[
\int_{Y}f\left(y\right)d\mu\left(y\right)=\int_{Y/H}\left(\int_{H}f\left(yh\right)d\mu_{H}\left(h\right)\right)d\mu_{Y/H}\left(Hy\right).
\]
\end{thm}

\begin{prop}
\label{prop: measure on spread model}Let $H$ be a Lie group acting
smoothly, almost freely and properly on a manifold $\manifold$ and
$F$ a spread model for $H$ in $\mathcal{M}$; let $\mu_{\manifold}$
be an $H$-invariant Radon measure on $\manifold$ and $\mu_{H}$
a Haar measure on $H$. Finally let $\mu_{\manifold/H}$ be the unique
Radon measure on $\manifold/H$ satisfying for every $f\in L^{1}\left(\manifold\right)$
\begin{equation}
\mu_{\manifold}\left(f\right)=\mu_{\manifold/H}\left(\int_{H}f\left(xh\right)d\mu_{H}\left(h\right)\right).\label{eq:1}
\end{equation}
Then, if $B\subseteq\manifold/H$ and $\Hbcs\subseteq H$ are BCS,
\[
\mu_{\manifold}\left(\manifold_{B}\cdot\Hbcs\right)=\mu_{\manifold/H}\left(B\right)\cdot\mu_{H}\left(B_{H}\right).
\]
\end{prop}

\begin{proof}
We clearly may assume that $B$ is small enough so that $\manifold_{B}$
is contained in a single $F_{\alpha}$ and $B_{H}\subseteq H_{0}$.
Let $f=1_{\manifold_{B}B_{H}}$. Every $x_{0}\in\mathcal{M}$ such
that $\pi\left(x_{0}\right)\in B$ can be written $x_{0}=y_{0}h_{0}$,
where $y_{0}\in F$ and $h_{0}\in H$. Hence $y_{0}h_{0}h\in\manifold_{B}B_{H}$
iff $y_{0}\in\manifold_{B}$ and $h_{0}h\in B_{H}$. As a result,
\[
\int_{H}1_{\manifold_{B}B_{H}}\left(y_{0}h_{0}h\right)d\mu_{H}\left(h\right)=1_{\manifold_{B}}\left(y_{0}\right)\mu_{H}\left(B_{H}\right)=1_{B}\left(\pi\left(x_{0}\right)\right)\mu_{H}\left(B_{H}\right).
\]
All in all, by \eqref{1}
\[
\mu\left(\manifold_{B}\cdot\Hbcs\right)=\mu_{\manifold/H}\left(B\right)\cdot\mu_{H}\left(B_{H}\right).
\]
\end{proof}

\subsection{Spread models for compact quotients}

Our main motivation for exploring the compact case is to find a spread
model for the sphere $\sphere{n-1}$, which is diffeomorphic $\so{n-1}\left(\RR\right)\backslash\so n\left(\RR\right)$,
inside $\so n\left(\RR\right)$. Indeed, in Secion \ref{sec:Special-examples-in-SL}
we use the results of this part to construct a spread model for the
sphere. 
\begin{prop}
\label{prop: K' is s.m. for sphere}Let $K$ be a Lie group. Assume
that $K^{\dprime}<K$ is a closed subgroup such that the quotient
space $K/K^{\dprime}$ is compact. Then, there exists a spread model
$K^{\prime}$ for $K/K^{\dprime}$.
\end{prop}

\begin{rem}
\label{rem: intersection and union of BCS}Since $\del\left(A\cup B\right),\del\left(A\cap B\right)\subseteq\del A\cup\del B$,
the union, intersection and subtraction of BCSs are in themselves
BCS. Also, a direct product of BCS's is a BCS in the direct product
of the manifolds, and a diffeomorphic image of a BCS is a BCS. 
\end{rem}

\begin{proof}[Proof of Proposition \ref{prop: K' is s.m. for sphere}]
Since $\pi:K\to K/K^{\dprime}$ is a fiber bundle with a fiber $K^{\dprime}$,
there exists an open covering $\left\{ U_{\alpha}\right\} $ of $K/K^{\dprime}$
such that for every $\alpha$ there are local sections $s_{\alpha}:U_{\alpha}\to K$.
We may assume that for each $U_{\alpha}$ there is an open subset
$W_{\alpha}$, whose closure lies inside $U_{\alpha}$, such that
$\left\{ W_{\alpha}\right\} $ is also a covering; the sets $W_{\a}$
can be chosen to be BCS's (e.g., by reducing to contained open balls);
then, by compactness, this covering can be made finite. Set $B_{\alpha}:=W_{\alpha}\setminus\cup_{i=1}^{\alpha-1}W_{i}$.
The sets $B_{\a}$ are disjoint, and they maintain the BCS property
(Remark \ref{rem: intersection and union of BCS}). It is therefore
clear that $V_{\alpha}:=s_{\alpha}\left(U_{\alpha}\right)$ and $F_{\alpha}:=s_{\alpha}\left(\interior{B_{\alpha}}\right)$
satisfy conditions 2 and 3 in Definition \ref{def: Spread Model}.
Define 
\[
\theta_{\alpha}:V_{\alpha}\times K^{\dprime}\to\pi^{-1}\left(U_{\alpha}\right)=V_{\alpha}K^{\dprime}
\]
 by 
\[
\theta_{\alpha}\left(u,h\right)=u\cdot h.
\]
This is clearly a diffeomorphism, so by Lemma \ref{lem:FHIsM}, the
first condition also holds. The last condition is fulfilled since
for a compact $B\subseteq K/K^{\dprime}$ we have that
\[
\pi^{-1}|_{\overline{F}}\left(B\right)=\cup_{\alpha}\pi^{-1}|_{\overline{F}}\left(B\cap B_{\alpha}\right)=\cup_{\alpha}\theta_{\alpha}^{-1}\left(\left(B\cap B_{\alpha}\right)\times H\right)\cap\overline{F}=\cup_{\alpha}\theta_{\alpha}^{-1}\left(\left(B\cap B_{\alpha}\right)\times1_{H}\right),
\]
which is clearly compact.
\end{proof}
\begin{rem}
\label{rem:spread models of subsubgroup}It is clear from the proof
that if $K^{\dprime}<K_{1}^{\dprime}$ are closed Lie subgroups of
a Lie group $K$ such that $K/K^{\dprime}$ is compact, then one can
find $F_{1}\subset F$ such that $F$ is a spread model of $K$ w.r.t.\
$K^{\dprime}$ and $F_{1}$ is a spread model of $K$ w.r.t.\ $K_{1}^{\dprime}$.
\end{rem}

\begin{cor}
\label{cor:BCS fundamental domain for finite groups in compact groups}Let
$K$ be a compact Lie group and let $\Gamma<K$ be a lattice (i.e.
a finite subgroup). There exists a fundamental domain $\mathcal{\mathcal{K}}\subseteq K$
for $\gam$ which is BCS. 
\end{cor}

\begin{proof}
This is a direct consequence of Proposition \ref{prop: K' is s.m. for sphere}. 
\end{proof}

\subsection{Manipulations on spread models}

The following result allows us to ``compose'' spread models in the
sense that a spread model for $\manifold/G$ ``times'' a spread
model for $G/H$ is a spread model for $\manifold/H$.
\begin{prop}
\label{prop: M/G * G/H =00003D M/H}Suppose that a Lie group $G$
acts smoothly,  freely and properly on a manifold $\manifold$. Let
$H$ be a closed subgroup of $G$. 
\begin{enumerate}
\item If $F^{\manifold}$ is a spread model of $\manifold/G$ in $\manifold$
and $F^{G}$ is a spread model of $G/H$ in $G$ (considered as a
manifold), then $F=F^{\manifold}\cdot F^{G}$ is a spread model of
$\manifold/H$ in $\manifold$.
\item If $B^{\manifold/G}\subset\manifold/G$ and $B^{G/H}\subset G/H$
are BCS's, then so does 
\[
B^{\manifold/H}:=\pi_{\manifold/H}^{\manifold}\left(\mathcal{B}^{\mathcal{M}}\cdot\mathcal{B}^{G}\right)\subset\manifold/H,
\]
where $\mathcal{B}^{\mathcal{M}},\mathcal{B}^{G}$ are the representatives
in $F^{\manifold},F^{G}$ of $B^{\manifold/G},B^{G/H}$ respectively. 
\item If $\mu_{G/H},\mu_{\manifold/G},\mu_{\manifold/H}$ are the measures
appearing in Proposition \ref{prop: measure on spread model}, then
\[
\mu_{\manifold/H}\left(B^{\manifold/H}\right)=\mu_{\manifold/G}\left(B^{\manifold/G}\right)\mu_{G/H}\left(B^{G/H}\right).
\]
\end{enumerate}
\end{prop}

\begin{proof}
We begin with the \textbf{first part}. Let $F_{\alpha}^{\manifold}\subseteq V_{\alpha}^{\manifold}\subseteq\manifold$
and $F_{\beta}^{G}\subseteq V_{\beta}^{G}\subseteq G$ be as in the
definition of a spread model. The sets $V_{\alpha\beta}=V_{\alpha}^{\manifold}V_{\beta}^{G}$
are embedded manifolds as they are the diffeomorphic image of $V_{\alpha}^{\manifold}\times V_{\beta}^{G}$
under the map $\theta_{\alpha}^{-1}$, which is defined on an open
submanifold of $\manifold$ (see Lemma \ref{lem:FHIsM}). Moreover,
the composition of the diffeomorphisms 
\[
\xymatrix{V_{\alpha\beta}\times H_{0}\ar[rr]^{\theta_{\alpha}^{-1}\times id_{H_{0}}} &  & V_{\alpha}^{\manifold}\times V_{\beta}^{G}\times H_{0}\ar[rr]^{id_{V_{\alpha}^{\manifold}}\times\theta_{\beta}} &  & V_{\alpha}^{\manifold}\times V_{\beta}^{G}H_{0}\ar[r]^{\theta_{\alpha}} & V_{\alpha\beta}H_{0}}
\]
is a diffeomorphism, and it is $\theta_{\alpha\beta}$. Since $V_{\beta}^{G}H_{0}$
is open in $G$, then $V_{\alpha}^{\manifold}\times V_{\beta}^{G}H_{0}$
is open in $V_{\alpha}^{\manifold}\times G$; so $V_{\alpha\beta}$
is open in $\manifold$ as the image of the open map $\theta_{\alpha}$.
The first property of a spread model is now established.

Let $F_{\alpha\beta}=F_{\alpha}^{\manifold}F_{\beta}^{G}=\theta_{\alpha\beta}\brac{F_{\alpha}^{\manifold}\times F_{\beta}^{G}}$.
Since $\theta_{\alpha\beta}$ is open $F_{\alpha\beta}$ is also open.
$\overline{F_{\alpha}^{\manifold}}\subseteq V_{\alpha}^{\manifold}$
and $\overline{F_{\beta}^{G}}\subseteq V_{\beta}^{\manifold}$, hence
$\theta_{\alpha\beta}\brac{\overline{F_{\alpha}^{\manifold}}\times\overline{F_{\beta}^{G}}}=\overline{F_{\alpha}^{\manifold}}\cdot\overline{F_{\beta}^{G}}\subseteq\overline{F_{\alpha\beta}}$.
As $\theta_{\alpha\beta}^{-1}$ is continuous we conclude that $\overline{F_{\alpha}^{\manifold}}\cdot\overline{F_{\beta}^{G}}$
is closed and hence $\overline{F_{\alpha}^{\manifold}}\cdot\overline{F_{\beta}^{G}}=\overline{F_{\alpha\beta}}$.
Furthermore, 
\[
\cup_{\alpha,\beta}\overline{F_{\alpha\beta}}=\cup_{\alpha,\beta}\overline{F_{\alpha}^{\manifold}}\cdot\overline{F_{\beta}^{G}}=\brac{\cup_{\alpha}\overline{F_{\alpha}^{\manifold}}}\cdot\brac{\cup_{\beta}\overline{F_{\beta}^{G}}}=\overline{F^{\manifold}}\cdot\overline{F^{G}}\supseteq F.
\]
Note for future use that this shows that $\overline{F^{\manifold}}\cdot\overline{F^{G}}$
is a closed set, hence $\overline{F^{\manifold}}\cdot\overline{F^{G}}\supseteq\overline{F}$
(in fact we get an equality). 

We check the injectivity of $\pi|_{F_{\alpha\beta}}$. Suppose that
$\pi\brac{p_{1}^{\mathcal{M}}p_{1}^{G}}=\pi\brac{p_{2}^{\mathcal{M}}p_{2}^{G}}$,
or in other words, $p_{1}^{\mathcal{M}}p_{1}^{G}H=p_{2}^{\mathcal{M}}p_{2}^{G}H$.
Hence modulo $G$ we have that $p_{1}^{\mathcal{M}}G=p_{2}^{\mathcal{M}}G$.
Since the projection to $\mathcal{M}/G$ is injective when restricted
to $F_{\alpha}^{\mathcal{M}}$, then $p_{1}^{\mathcal{M}}=p_{2}^{\mathcal{M}}$.
Since the $G$ action on $\mathcal{M}$ is free, we conclude that
$p_{1}^{G}H=p_{2}^{G}H$ -{}- but this forces $p_{1}^{G}=p_{2}^{G}$,
by the injectivity of the projection to $G/H$ restricted to $F_{\beta}^{G}$.

Next we show that $\del_{V_{\alpha\beta}}F_{\alpha\beta}$ is contained
in a finite union of lower dimensional submanifolds of $\manifold$.
This follows from the analogous assumptions on $F_{\alpha}^{\manifold}$
and $F_{\beta}^{G}$, along with the following: 
\[
\del_{V_{\alpha\beta}}F_{\alpha\beta}=\left(\brac{id_{V_{\alpha}^{\manifold}}\times\theta_{\beta}^{-1}}\circ\theta_{\alpha}^{-1}\right)\left(\del\brac{F_{\alpha}^{\manifold}\times F_{\beta}^{G}}\times\left\{ e\right\} \right).
\]

Now we show that $\pi|_{\overline{F}}$ is a proper map. Suppose that
$p_{n}\to\infty$, where $p_{n}\in\overline{F}$. Since $\overline{F}=\overline{F^{\manifold}}\cdot\overline{F^{G}}$,
we may decompose $p_{n}=p_{n}^{\mathcal{M}}p_{n}^{G}$. Assume by
contradiction that $p_{n}H$ does not tend to infinity. Consequentially,
there exist $h_{n_{k}}\in H$ and $p\in\mathcal{M}$ such that $p_{n_{k}}h_{n_{k}}\to p$.
For convenience we will write instead that $p_{n}h_{n}\to p$. Since
$\overline{F}\subseteq\overline{F^{\mathcal{M}}}G\simeq\overline{F^{\mathcal{M}}}\times G$,
we can find some $\alpha$ such that, for almost every $n$, $p_{n}h_{n}\in V_{\alpha}^{\mathcal{M}}G$.
We clearly have $p_{n}^{\mathcal{M}}\to p^{\mathcal{M}}$ and $p_{n}^{G}h_{n}\to p^{G}$,
where $p=p^{\mathcal{M}}p^{G}$. Since $G$ is assumed to act properly
on $\mathcal{M}$, we conclude from $p_{n}\to\infty$ that $p_{n}^{G}\to\infty$.
However, the restriction to $F^{G}$ of the projection $G\to G/H$
is assumed to be proper, so we get a contradiction with $p_{n}^{G}h_{n}\to p^{G}$.

Finally, we need to check that $F$ is a full set of representatives.
Let $x\in\manifold/H$ and set $x^{\prime}=xG\in\manifold/G$. There
is $p^{\mathcal{M}}\in F^{\mathcal{M}}$ which projects to $x^{\prime}$.
Since $p^{\mathcal{M}}H$ and $x$ both project to $x^{\prime}$,
there is some $g\in G$ such that $p^{\mathcal{M}}gH=x$. By definition
of $F^{G}$, there is $p^{G}\in F^{G}$ which projects to $gH$. As
a result, $p^{\mathcal{M}}p^{G}H=x$ i.e.\ $\pi\left(p^{\mathcal{M}}p^{G}\right)=x$.

We turn to the \textbf{second part} of the Proposition. We know from
the first part that $\pi_{\manifold/G}^{\manifold}\left(F_{\alpha}^{\manifold}\right)$
and $\pi_{G/H}^{G}\left(F_{\beta}^{G}\right)$ are open submanifolds
of $\manifold/G$ and $G/H$ respectively, and that the maps 
\[
\rho_{\alpha\beta}:\pi_{\manifold/G}^{\manifold}\left(F_{\alpha}^{\manifold}\right)\times\pi_{GH}^{G}\left(F_{\beta}^{G}\right)\to F_{\alpha}^{\manifold}\times F_{\beta}^{G}\to F_{\alpha\beta}\to\manifold/H
\]
are open diffeomorphisms onto their image. Denote $L_{\alpha\beta}:=\overline{\pi_{\manifold/G}^{\manifold}\left(F_{\alpha}^{\manifold}\right)}\times\overline{\pi_{GH}^{G}\left(F_{\beta}^{G}\right)}$.
Similarly to the proof of Proposition \ref{prop: BCS from manifold to space},
we split each 
\[
B_{\alpha\beta}:=\left(B^{\manifold/G}\times B^{G/H}\right)\cap L_{\alpha\beta}
\]
 into two parts: $B\cap\del L_{\alpha\beta}$ and $B_{\text{int}}:=B\cap L_{\alpha\beta}^{\circ}$.
Since, by part (2) of Lemma \ref{lem:boundary calculations}, 
\[
\del\rho_{\alpha\beta}\left(B\right)\subseteq\rho_{\alpha\beta}\left(\del L_{\alpha\beta}\right)\cup\del_{\rho_{\alpha\beta}\left(L_{\alpha\beta}^{\circ}\right)}\rho_{\alpha\beta}\left(B_{\text{int}}\right),
\]
 it is enough to show that $\rho_{\alpha\beta}\left(B_{\text{int}}\right)$
is BCS inside $\rho_{\alpha\beta}\left(L_{\alpha\beta}^{\circ}\right)$
and that locally $\rho_{\alpha\beta}\left(\del L_{\alpha\beta}\right)$
is contained in a finite union of codimension $\geq1$ submanifolds.

The first statement is clear since $\rho_{\alpha\beta}|_{L_{\alpha\beta}^{\circ}}$
is an open diffeomorphism onto its image $\rho_{\alpha\beta}\left(L_{\alpha\beta}^{\circ}\right)$. 

For the second statement, first notice that 
\[
\rho_{\alpha\beta}\left(\del L_{\alpha\beta}\right)=\pi_{\manifold/H}^{\manifold/H_{0}}\circ\pi_{\manifold/H_{0}}^{\manifold}\left(\theta_{\alpha}^{-1}\left(\del_{V_{\alpha}\times V_{\beta}}\left(F_{\alpha}\times F_{\beta}\right)\right)\right).
\]
By assumption, $\del_{V_{\alpha}\times V_{\beta}}\left(F_{\alpha}\times F_{\beta}\right)$
is locally contained in a finite union of codimension $\geq1$ submanifolds
(w.r.t. $V_{\alpha}\times V_{\beta}$). Hence, using the fact that
$\pi_{\manifold/H_{0}}^{\manifold}\circ\theta_{\alpha}^{-1}$ is an
open diffeomorphism, we conclude that $\pi_{\manifold/H_{0}}^{\manifold}\left(\theta_{\alpha}^{-1}\left(\del_{V_{\alpha}\times V_{\beta}}\left(F_{\alpha}\times F_{\beta}\right)\right)\right)$
is also locally contained in a finite union of codimension $\geq1$
submanifolds w.r.t. $\manifold/H_{0}$. Finally, this property is
stable under $\pi_{\manifold/H}^{\manifold/H_{0}}$ since this map
is a local diffeomorphism. 

It is left to prove the \textbf{third part} of the proposition. Let
$\mu_{G},\mu_{H},\mu_{\mathcal{M}}$ be the measures appearing in
Proposition \ref{prop: measure on spread model}, and assume $A\subset H$
be a compact neighborhood of $e_{H}$ so that $0<\mu_{H}\left(A\right)<\infty$.
By Proposition \ref{prop: measure on spread model} we have the following
equalities:
\begin{align*}
\mu_{\manifold}\left(\mathcal{B}^{\manifold}\mathcal{B}^{G}A\right) & =\mu_{G}\left(\mathcal{B}^{G}A\right)\mu_{\manifold/G}\left(B^{\manifold/G}\right),\\
\mu_{G}\left(\mathcal{B}^{G}A\right)= & \mu_{G/H}\left(B^{G/H}\right)\mu_{H}\left(A\right),\\
\mu_{\manifold}\left(\mathcal{B}^{\manifold}\mathcal{B}^{G}A\right) & =\mu_{H}\left(A\right)\mu_{\manifold/H}\left(B^{\manifold/H}\right).
\end{align*}
It follows that
\[
\mu_{\manifold/H}\left(B^{\manifold/H}\right)=\mu_{\manifold/G}\left(B^{\manifold/G}\right)\mu_{G/H}\left(B^{G/H}\right).
\]
\end{proof}
The content of the following proposition is that if a space can be
written as a quotient in two ways, $\mathcal{M}/G$ and $\mathcal{\manifold}^{\prime}/H$,
where $\mathcal{\manifold}^{\prime}\subset\manifold$ and $H<G$,
then a spread model for the latter a also a spread model for the first. 
\begin{prop}
\label{prop: S.M. for space that can be written as two quotients}Suppose
that $\mathcal{M}$ is a manifold and $\mathcal{M}^{\prime}$ is an
embedded closed submanifold of $\mathcal{M}$. Suppose that $G$ is
a Lie group and $H$ is a closed subgroup of $G$ satisfying $H_{0}=G_{0}\cap H$.
Assume that $G$ acts on $\mathcal{M}$ and $H$ is stabilizes $\mathcal{M}^{\prime}$;
the action is free, proper and smooth. Finally assume that the map
$\iota:\mathcal{M}^{\prime}/H\to\mathcal{M}/G$ given by $m^{\prime}H\to mG$
is a diffeomorphism. If $F$ is a spread model of $\mathcal{M}^{\prime}/H$
in $\mathcal{M}^{\prime}$, then it is also a spread model of $\mathcal{M}/G$
in $\mathcal{M}$.
\end{prop}

\begin{proof}
Let $V_{\alpha},F_{\alpha}\subseteq\mathcal{M}^{\prime}$ as in the
definition of a spread model. The only condition that is not clear
is the first one. By Lemma \ref{lem:FHIsM}, it is sufficient to prove
that $\theta_{\alpha}:V_{\alpha}\times G_{0}\to V_{\alpha}G_{0}$
is a diffeomorphism whose image is open. For this, one has to show
that $\theta_{\alpha}$ is an injective immersion (since $V_{\alpha}\times G_{0}$
and $\mathcal{M}$ both have the same dimension). 

If $pg=p_{1}\in V_{\alpha}G_{0}$, then $p=p_{1}$ modulo $G_{0}$.
Since $p,p_{1}\in V_{\alpha}$, we conclude that $p=p_{1}$ modulo
$H$. In other words, there is $h\in H$ such that $ph=p_{1}$. Since
the action is free, we get that $g=h\in G_{0}\cap H=H_{0}$. This
forces $p=p_{1}$ and $g=h=e$, since $\theta_{\alpha}^{\prime}:V_{\alpha}\times H_{0}\to V_{\alpha}H_{0}$
is injective. 

The natural map $\pi_{\alpha}^{0}:V_{\alpha}\to\mathcal{M}/G_{0}$
is an open diffeomorphism onto its image due to the following argument.
Consider the following commutative diagram:
\[
\xymatrix{V_{\alpha}\ar[r]\ar[rd]^{\pi_{\alpha}^{0}} & \mathcal{M}^{\prime}/H_{0}\ar[d]^{proj}\ar[r]^{\phi} & \mathcal{M}^{\prime}/H\ar[d]^{\iota}\\
 & \mathcal{M}/G_{0}\ar[r]^{\iota^{0}} & \mathcal{M}/G
}
.
\]
Since $\pi_{\alpha}^{0},\phi$ and $\iota$ are immersions, we conclude
that $\iota^{0}\circ proj\circ\pi_{\alpha}^{0}$ is also an immersion.
Since $\pi_{\alpha}^{0}$ is an open diffeomorphism by assumption,
we get that $proj$ must be an immersion. If $p_{1},p_{2}\in\mathcal{M}^{\prime}$
such that $p_{2}=p_{1}g$ for some $g\in G_{0}$, it must be that
$p_{1}$ and $p_{2}$ are equal modulo $H$, since $\iota$ is injective.
Freeness of the action implies that $g\in H$, and so $g\in H_{0}$.
As a result, $proj$ is injevtive and so $proj\circ\pi_{\alpha}^{0}$
is also. The claim now follows since $V_{\alpha}$ and $\mathcal{M}/G_{0}$
have the same dimension.
\end{proof}

\section{Construction of fundamental domains for $\protect\sl m\left(\protect\ZZ\right)$}

The goal of this section is to recall a construction for fundamental
domains of $\sl m\left(\ZZ\right)$ inside $\sl m\left(\RR\right)$
and inside $\so m\left(\RR\right)\backslash\sl m\left(\RR\right)$.
The motivation for this is that in the next section we will see that
these fundamental domains are spread models for two spaces of lattices
that we now turn to describe. 

Let $\Lambda$ be a (full rank) lattice inside $\RR^{m}$ having the
columns of $\mtx\in\gl m\left(\RR\right)$ as an ordered basis. It
is well known that any other basis of $\Lambda$ appears as the columns
of a matrix obtained by multiplying $\mtx$ from the right by a matrix
from $\gl m\left(\ZZ\right)$. As a result, the space of $m$-lattices
can be defined as $\gl m\left(\RR\right)/\gl m\left(\ZZ\right)$.
One can also consider a more crude space which is the space of \emph{shapes}
of lattices: two lattices $\Lambda_{1}=\mtx_{1}\cdot\gl m\left(\ZZ\right)$
and $\Lambda_{2}=\mtx_{2}\cdot\gl m\left(\ZZ\right)$ have the same
shape if $\Lambda_{1}$ differs from $\Lambda_{2}$ by an orthogonal
transformation and rescaling, namely there are $k\in\ort m\left(\RR\right)$
and $c>0$ such that $ck\mtx_{1}\cdot\gl m\left(\ZZ\right)=\mtx_{2}\cdot\gl m\left(\ZZ\right)$.
As a result, the space of shapes can be defined as 
\[
\shapespace m=\po m\left(\RR\right)\backslash\pgl m\left(\RR\right)/\pgl m\left(\ZZ\right)\backsimeq\so m\left(\RR\right)\backslash\sl m\left(\RR\right)/\sl m\left(\ZZ\right).
\]
Notice that in the right hand side we consider unimodular lattices
(i.e. lattices with covolume $1$), since clearly every lattice can
be rescaled to a unimodular lattice. We let 
\[
\latspace m:=\sl m\left(\RR\right)/\sl m\left(\ZZ\right)
\]
denote the space of unimodular latices in $\RR^{m}$. 

In Subsection \ref{subsec: Siegel reduced bases} we introduce a variant
of Siegel sets inside $\sl m\left(\RR\right)$, which contain a finite
number of representatives from every (right) orbit of $\sl m\left(\ZZ\right)$,
and therefore a fundamental domain; in Subsection \ref{subsec: constructing F_m}
we define the fundamntal domain $\groupfund m\subset\sl m\left(\RR\right)$
inside the Siegel set, as well as the resulting fundamental domain
$\symfund m$ inside $\so m\left(\RR\right)\backslash\sl m\left(\RR\right)$,
or more precisely in $P_{m}$, the group of upper triangular matrices
of determinant $1$ with positive diagonal entries; indeed the Iwasawa
decomposition of $\sl m\left(\RR\right)$ implies that $\so m\left(\RR\right)\backslash\sl m\left(\RR\right)$
is diffeomorphic to $P_{m}$.

Let us stress on the fact that the construction for $\symfund m$
is not new (\cite{Schmidt_98}, \cite{Grenier_93}), but we bring
it here for completeness, as well as for adding the way to obtain
$\groupfund m$ from $\symfund m$.

\subsection{Reduced bases\label{subsec: Siegel reduced bases}}

Write $P_{m}=A_{m}N_{m}$ where $A_{m}$ is the subgroup of diagonal
matrices, and $N_{m}$ the subgroup of unipotent matrices. 

Let $\Lambda$ be a lattice in $\RR^{m}$ (not necessarily unimodular).
We describe an inductive method to construct an ordered basis $\left\{ v_{1},\ldots,v_{m}\right\} $
for $\lat$ as follows. Let $v_{1}$ be a shortest nonzero element
of $\lat$. For future reference, we denote its length by $a_{1}$
and its direction $v_{1}/a_{1}$ by $\phi_{1}$. Next, write $V_{1}$
for $\sp{\RR}{v_{1}}_{\RR}$ and consider the projection of $\Lambda$
to $V_{1}^{\perp}$, which is a lattice of dimension $m-1$. One can
find a vector $v_{2}\in\Lambda$ whose projection to $V_{1}^{\perp}$
is of nonzero minimal length $a_{2}$. Since actually all the elements
in $\left\{ v_{2}+nv_{1}:n\in\ZZ\right\} $ share this property of
having their projection to $V_{1}^{\perp}$ be of length $a_{2}$,
we may assume that $v_{2}$ also satisfies that its projection to
$V_{1}$ is $n_{1,2}a_{1}\cdot\phi_{1}$ with $\left|n_{1,2}\right|\leq\frac{1}{2}$.
We proceed by induction: 
\begin{defn}[\textbf{and notations}]
\label{def: Siegel reduced basis}A \emph{Reduced} basis for a lattice
$\lat$ is a basis $\left\{ v_{1},\ldots,v_{m}\right\} $ in which
for all $j\in\left\{ 1,\ldots,m\right\} $, the basis element $v_{j}$
is chosen such that:

\begin{enumerate}
\item The projection of $v_{j}$ to $V_{j-1}^{\perp}=\perpen{\brac{\sp{\RR}{v_{1},\ldots,v_{j-1}}}}$
has minimal non-zero length $a_{j}$ (here $V_{0}=\left\{ 0\right\} $);
denote this projection by $a_{j}\phi_{j}$, where $\phi_{j}$ is a
unit vector. 
\item The projection of $v_{j}$ to $V_{j-1}=\sp{\RR}{v_{1},\ldots,v_{j-1}}=\sp{\RR}{\phi_{1},\ldots,\phi_{j-1}}$
is 
\[
\sum_{i=1}^{j-1}\exd{\left(n_{i,j}a_{i}\right)}{\mbox{scalars}}\phi_{i}\mbox{ with \ensuremath{\left|n_{ij}\right|\leq\frac{1}{2}\;\,}for all \ensuremath{i=1,\ldots,j-1}}.
\]
\end{enumerate}
The matrix $\mtx=\left[\begin{array}{ccc}
v_{1} & \cdots & v_{m}\end{array}\right]$ is called a \emph{reduced matrix} of $\Lambda$.
\end{defn}

We note that in the case of unimodular bases (bases of co-volume $1$),
one may need to replace $v_{1}$ by $-v_{1}$ in order for the reduced
matrix $\mtx$ to have determinant $1$ (and not $-1$). 

The parameters $\left\{ a_{j}\right\} $, $\left\{ n_{i,j}\right\} $
and $\left\{ \phi_{j}\right\} $ involved in the process of constructing
a reduced basis $\left\{ v_{1},\ldots,v_{m}\right\} $ are interpreted
via the $KAN$ coordinates of the associated reduced matrix as follows.
Let 
\[
a=\diag{a_{1},\ldots,a_{m}},\;k=\left[\begin{matrix}\phi_{1} & \cdots & \phi_{m}\end{matrix}\right]
\]
and
\[
n=\begin{pmatrix}1 & n_{1,1} & \dots & n_{1,m}\\
 & 1 &  & \vdots\\
 &  & \ddots & n_{m-1,m}\\
 &  &  & 1
\end{pmatrix}.
\]
Then, since the $i$-th column of $ka$ is $a\phi_{i}=$ the projection
of $v_{i}$ to $V_{i-1}^{\perp}$, and the $i$-th column of $n$
is exactly the coordinates of $v_{i}$ w.r.t.\ the orthogonal set
$\left\{ a_{1}\phi_{1},\dots,a_{m}\phi_{m}\right\} $, we obtain that
the reduced matrix is 
\[
\mtx=\left[\begin{array}{ccc}
v_{1} & \cdots & v_{m}\end{array}\right]=kan.
\]

\begin{lem}
\label{lem: BLC. facts about z in RS domain}Suppose $\mtx=k_{_{M}}a_{_{M}}n_{_{M}}$
is -reduced w.r.t.\ some lattice $\Lambda$, where $k_{_{M}}$, $a_{_{M}}$
and $n_{_{M}}$ are as above. 

\begin{enumerate}
\item \label{enu: entries of n in =00005B-0.5,0.5=00005D}$n_{\mtx}$ is
a unipotent upper triangular matrix whose entries are bounded in $\left[-\frac{1}{2},\frac{1}{2}\right]$
(in particular, $\norm{n_{_{M}}^{\pm1}},\norm{n_{_{M}}^{\pm\transpose}}\porsmall1$).
\item \label{enu: entries of a increasing}$a_{_{M}}=\diag{a_{1},\ldots,a_{m}}$
is a diagonal matrix which satisfies that $a_{1}\porsmall\cdots\porsmall a_{m}$.
Specifically, $\frac{\sqrt{3}}{2}a_{j}\leq a_{j+1}$. 
\item \label{enu: norm out of E_(j-1)}If $\lm\in\lat$ (i.e. $\lm=\mtx v$
for some $v\in\mathbb{Z}^{m}$) satisfies $\lm\notin V_{j-1}$, then
\[
\left\Vert \lm\right\Vert \geq\text{dist}\left(\lm,V_{j-1}\right)\geq\dist{v_{j},V_{j-1}}=a_{j}.
\]
\item \label{enu: norm in E_j}If $x\in V_{j}$, then $\left\Vert a_{_{M}}x\right\Vert \porsmall a_{j}\left\Vert x\right\Vert $. 
\end{enumerate}
\end{lem}

\begin{proof}
Parts \ref{enu: entries of n in =00005B-0.5,0.5=00005D} and \ref{enu: norm out of E_(j-1)}
are immediate from the construction of $\mtx$. For part \ref{enu: entries of a increasing},
recall that $\left\{ a_{1}\phi_{1},\dots,a_{m}\phi_{m}\right\} $
is an orthogonal set in $\RR^{m}$ and that  $v_{j+1}=a_{j+1}\phi_{j+1}+\sum_{i=1}^{j}n_{i,j+1}a_{i}\phi_{i}$.
Then,
\[
a_{j}^{2}\overset{\substack{_{\mbox{part \ref{enu: norm out of E_(j-1)}}}\\
_{(v_{j+1}\notin V_{j-1})}
}
}{\leq}\text{dist}\left(v_{j+1},V_{j-1}\right)^{2}=\norm{n_{j,j+1}\left(a_{j}\phi_{j}\right)+\left(a_{j+1}\phi_{j+1}\right)}^{2}=a_{j}^{2}\left|n_{j,j+1}\right|^{2}+a_{j+1}^{2}.
\]
Now, since $\left|n_{j,j+1}\right|\leq\frac{1}{2}$ (by part \ref{enu: entries of n in =00005B-0.5,0.5=00005D}),
we obtain:
\[
\frac{1}{4}a_{j}^{2}+a_{j+1}^{2}\geq a_{j}^{2}
\]
and therefore 
\[
a_{j+1}\geq\frac{\sqrt{3}}{2}a_{j}.
\]
As for part \ref{enu: norm in E_j}, 
\[
\norm{a_{_{\text{}}}x}=\bignorm{\left(\begin{smallmatrix}a_{1}\\
 & \ddots\\
 &  & a_{m}
\end{smallmatrix}\right)\left(\begin{smallmatrix}x_{1}\\
\vdots\\
x_{j}\\
0
\end{smallmatrix}\right)}=\bignorm{\left(\begin{smallmatrix}a_{1}x_{1}\\
\vdots\\
a_{j}x_{j}\\
0
\end{smallmatrix}\right)}\leq\max_{1\leq i\leq j}\left|a_{i}\right|\norm x\overset{_{\mbox{part \ref{enu: entries of a increasing}}}}{\asymp}a_{j}\norm x.
\]
\end{proof}
\begin{defn}
\label{def: Siegel-Karasik set}We refer to the sets
\[
\left\{ \mtx\in\gl m\left(\RR\right):\:\substack{\mbox{\ensuremath{\mtx} is reduced for the lattice }\\
\mbox{spanned by its columns}
}
\right\} 
\]
as \emph{reduced Siegel sets}.
\end{defn}

\begin{rem}
The reduced Siegel sets are contained in the well-known Siegel sets
(e.g., \cite[Chapter X]{Bekka_Mayer,Rag_72}). 
\end{rem}

We note that parts \ref{enu: entries of n in =00005B-0.5,0.5=00005D}
and \ref{enu: norm out of E_(j-1)} of Lemma \ref{lem: BLC. facts about z in RS domain}
are the defining conditions of the reduced Siegel sets (the inequalities
in parts \ref{enu: entries of a increasing} and \ref{enu: norm in E_j}
are redundant, since they follow from part \ref{enu: norm out of E_(j-1)}).
Observe that these defining inequalities depend only on the entries
of the $N$ and $A$ components of the matrix. Indeed, let $\mtx=\left[\begin{array}{ccc}
v_{1} & \cdots & v_{m}\end{array}\right]=kan$ and $z=an$; the inequalities in \ref{enu: entries of n in =00005B-0.5,0.5=00005D}
are on the entries of $n$, and the inequalities in \ref{enu: norm out of E_(j-1)}
translate into 
\begin{equation}
a_{j}\leq\bignorm{\substack{\mbox{projection of \ensuremath{zv}}\\
\mbox{to }\sp{\RR}{e_{j},\ldots e_{m}}
}
},\label{eq:dist}
\end{equation}
for every $v=\left(\alpha_{1},\dots,\alpha_{m}\right)^{\transpose}\in\ZZ^{n}$
and $j=1,\ldots,m$. This is because: 
\[
a_{j}=\dist{v_{j},V_{j-1}}\leq\dist{\exd{\sum_{i=j}^{m}\alpha_{i}v_{i}}{\mtx v},V_{j-1}}\overset{_{\mbox{rotation by \ensuremath{k}}}}{=}\dist{\sum_{i=j}^{m}\alpha_{i}z_{i},E_{j-1}}
\]
where $E_{j-1}:=\sp{\RR}{e_{1},\ldots,e_{j-1}}$ and $z=\left[z_{1},\ldots,z_{m}\right]$,
\[
=\dist{zv,E_{j-1}}=\norm{\substack{\mbox{projection of \ensuremath{zv}}\\
\mbox{to }\sp{\RR}{e_{j},\ldots e_{m}}
}
}.
\]

We also note that the number of the defining inequalities for these
reduced Siegel sets is infinite: indeed, every $v\in\ZZ^{n}$ yields
an inequality in formula \ref{eq:dist} (resp. part \ref{enu: norm out of E_(j-1)}
of Lemma \ref{lem: BLC. facts about z in RS domain}). However, it
is shown in \cite[p.49]{Schmidt_98} that it is actually sufficient
to consider the inequalities \ref{eq:dist} for only \emph{finitely
many} $v\in\ZZ^{n}$. We state it for future reference:
\begin{prop}
\label{prop: RS sets defd by finite number of inequalites}The reduced
Siegel sets are defined by a finite number of inequalities in (the
entries of) the $N$ and $A$ components of a matrix.
\end{prop}

\begin{rem}
The finitely many integral vectors $u$ that imply the sufficient
inequalities from \ref{eq:dist} are as follows. For any $j=1,\ldots,m$
one considers the $v\in\ZZ^{n}$ which satisfy 
\[
\max\left(\left|\alpha_{j}\right|,\dots,\left|\alpha_{m}\right|\right)\leq\frac{C}{a_{j}}\norm{\substack{\mbox{projection of \ensuremath{zv}}\\
\mbox{to }\sp{\RR}{e_{j},\ldots e_{m}}
}
}
\]
where $C$ is some constant that depends only on $m$ and can be computed
explicitly from \cite{Schmidt_98}; clearly, there is only a finite
number of integral vectors $u$ which satisfy this condition. In particular,
the reduced Siegel sets can be computed explicitly. 
\end{rem}

\subsection{Fundamental domains of $\protect\sl m\left(\protect\ZZ\right)$ inside
$\protect\sl m\left(\protect\RR\right)$ and $P_{m}$\label{subsec: constructing F_m}}

From now on we shall only consider unimodular lattices (resp.\ bases)
in $\RR^{m}$. Since these unimodular lattices are identified with
cosets in $\sl m\left(\RR\right)/\sl m\left(\ZZ\right)$, where a
representative for a coset is a choice of a basis for the corresponding
lattice, it follows that a fundamental domain for (the right action
of) $\sl m\left(\ZZ\right)$ inside $\sl m\left(\RR\right)$ consists
of a unique choice of a basis for every unimodular lattice $\lat<\RR^{m}$. 

We know (by the construction presented in the previous subsection)
that every lattice has a reduced basis, and therefore the reduced
Siegel set  contains a fundamental domain for $\sl m\left(\ZZ\right)$.
We turn to describe a \emph{closure} of such a fundamental domain,
namely a choice of a unique reduced basis for almost every unimodular
lattice $\lat<\RR^{m}$.
\begin{rem}
\label{rem: finitely many SR bases}It is shown in \cite{Schmidt_98}
that the number of reduced bases for a lattice $\RR^{m}$ is finite,
where a bound on this number depends only on $m$, and not on the
lattice. Intuitively, this is due to the fact that a given lattice
has only a finite number of shortest vectors (where the bound on this
number depends only on the dimension), and a reduced basis is constructed
such that in every step, one chooses a shortest vector from some lattice. 
\end{rem}

Given a reduced basis $\left\{ v_{1},\ldots,v_{m}\right\} $, one
can clearly obtain further reduced bases for the same lattice by alternating
the signs of the elements $v_{j}$. Note that the corresponding reduced
matrices $\mtx$ will have the same $A$ components, and in fact they
will vary from each other only by the signs of the entries of $n$
and $k$ as follows:
\begin{itemize}
\item replacing $\left\{ v_{1},\ldots,v_{m}\right\} $ by $\left\{ -v_{1},\ldots,-v_{m}\right\} $
is done by multiplying from $\mtx$ from the left by $-I$ (replacing
$k$ by $-k$); 
\item replacing $v_{j}$ by $-v_{j}$ for $j=2,\ldots,m$ is done by changing
the sign of the $j$-th row and column of $n$ (above the diagonal)
and changing the sign of the $j$-th column of $k$, $\phi_{j}$. 
\end{itemize}
In order to preserve the property $\det\left(\mtx\right)=1$, one
is only allowed to alternate the sign of an \emph{even} number among
the $v_{j}$'s. In particular, one is allowed to change all signs
simultaneously if and only if $-I\in K=\so m\left(\RR\right)$. 
\begin{defn}
\label{def: Fund dom F_m}We let $\groupfund m\subset\sl m\left(\RR\right)$
denote the closed subset of the reduced Siegel set (Definition \ref{def: Siegel-Karasik set})
which satisfies the following conditions on the $N$ and  $K$ components:

\begin{enumerate}
\item \label{eq: signs of n}Condition on the sign of the first row of $n$:
\[
\begin{array}{c}
2|m\Longrightarrow n_{1,j}\geq0\text{ for \ensuremath{j>2}}\Longrightarrow n_{1,j}\in\left[0,\frac{1}{2}\right]\text{ for \ensuremath{j>2}}\\
2\nmid m\Longrightarrow n_{1,j}\geq0\text{ for \ensuremath{j>1}}\Longrightarrow n_{1,j}\in\left[0,\frac{1}{2}\right]\text{ for \ensuremath{j>1}}
\end{array}.
\]
\item \textcolor{green}{\label{enu:condition on K}}Condition on the $K$-components:
they lie in a closure of a fundamental domain of the lattice $Z\left(K\right)$
in $K$.
\end{enumerate}
We also denote by $\symfund m\subseteq P_{m}$ the projection of $\groupfund m$
mod $K$, namely the set of upper triangular matrices whose columns
are reduced bases for the lattice spanned by their columns, and whose
$N$ components satisfy condition \ref{eq: signs of n}. 
\end{defn}

We state the following for future reference:
\begin{cor}
\label{cor:boundary of Fn}The boundary of $\groupfund m$ (resp.
$\symfund m$) is contained in a finite union of lower-dimeansional
manifolds in $\sl m\left(\RR\right)$ (resp. $P_{m}$).
\end{cor}

\begin{proof}
According to Proposition \ref{prop: RS sets defd by finite number of inequalites}
 and Definition \ref{def: Fund dom F_m}, $\groupfund m$ and $\symfund m$
are defined by a finite number of polynomial inequalities. 
\end{proof}

The significance of $\groupfund m$ and $\symfund m$ stems from the
following:
\begin{thm}
\label{thm: Fundamental domain for SL(m,Z)} $\groupfund m\subset\sl m\left(\RR\right)$
and $\symfund m\subset P_{m}$ are the closures of fundamental domains
for the right actions of $\sl m\left(\ZZ\right)$ on $\sl m\left(\RR\right)$
and on $P_{m}$ respectively. 
\end{thm}

\begin{proof}
Since $\groupfund m$ is a subset of the reduced Siegel set, the
latter containing a basis for every lattice, defined by the additional
conditions from Definition \ref{def: Fund dom F_m}, which merely
impose a choice of signs for the basis elements --- we conclude that
$\groupfund m$ contains a representative (basis) for every unimodular
lattice in $\RR^{m}$. 

It is now sufficient to show that every $\mtx\in\interior{\groupfund m}$
is the unique representative for the lattice spanned by its columns.
In other words, if a given lattice has more than one representative
in $\groupfund m$, then these representatives (that are reduced matrices
for the lattice) lie in the boundary of $\groupfund m$. 

Let $\mtx\in\interior{\groupfund m}$. Then, $\mtx$ satisfies a strict
version of the inequalities in Formula \ref{eq:dist} and in Definition
\ref{def: Fund dom F_m}. Write $\mtx=kan$. Using induction and (the
strict version of) inequality \ref{eq:dist}, it is clear that $v_{1}$
is uniquely determined up to a sign; $v_{2}$ is uniquely determined
up to a sign and modulo $V_{1}$; and so forth, every $v_{j}$ is
uniquely determined up to a sign and modulo $V_{j-1}$. As a result,
$a$ is uniquely determined, and the columns of $k$, $\phi_{1},\dots,\phi_{m}$,
are determined up to a sign. Since $v_{j}=\sum_{i=1}^{j}n_{i,j}\left(a_{i}\phi_{i}\right)$,
one can show using reverse induction (from $i=j-1$ to $i=1$) that
there are unique $n_{i,j}$ satisfying the strict version of condition
\ref{eq: signs of n} from Definition \ref{def: Fund dom F_m} so
that $v_{1},\dots,v_{m}$ are determined up to a sign. According to
the explanation in the beginning of this section, the inequalities
in Definition \ref{def: Fund dom F_m} impose a unique choice of signs,
and therefore a unique representative in the interior of $\groupfund m$. 
\end{proof}
\begin{rem}[\textbf{Declaring abuse of notation}]
As mentioned in Theorem \ref{thm: Fundamental domain for SL(m,Z)},
$\groupfund m$ and $\symfund m$ are \emph{closures} of fundamental
domains; in order to obtain actual fundamental domains, one should
remove parts of their boundaries, leaving a unique representative
for every lattice. However, we will abuse notation and denote $\groupfund m$
and $\symfund m$ for the \emph{actual fundamental domains}, and not
their closures.
\end{rem}

\begin{notation}
\label{def: Notation for K_z}For a matrix $\mtx$, denote by $K_{M}$
a fundamental domain in $\so m\left(\RR\right)$ for the finite group
$\sym^{+}\left(M\right)$, which is the stabilizer in $\so m\left(\RR\right)$
of the lattice spanned by the columns of $M$. 
\end{notation}

\begin{prop}
\label{prop: F of G'' from F of H}The relation between the fundamental
domains $\groupfund m$ and $\symfund m$ is given by
\[
\groupfund m=\bigcup_{z\in\symfund m}K_{z}\cdot z,
\]
and when $z\in\interior{\symfund m}$ it holds that $\sym^{+}\left(M\right)=Z\left(K\right)$,
the center of $K=\so m\left(\RR\right)$. . 
\end{prop}

\begin{rem}
\label{rem: symmetric lattices are in the boundary}In fact, it is
shown in \cite[p.50]{Schmidt_98} that the interior of \emph{any}
fundamental domain of $\sl m\left(\ZZ\right)$ consists of matrices
$M$ for which the $\sym^{+}\left(M\right)$ is $Z\left(K\right)$. 
\end{rem}

\begin{cor}
From Proposition \ref{prop: F of G'' from F of H} it follows that
the measure of $\groupfund m$ is the measure of $\symfund m$ times
the measure of a generic fiber in $\so m\left(\RR\right)$ (namely
the fiber of the interior points), which is the measure of $\so m\left(\RR\right)$
divided by the index of its center: $2$ when $m$ is even, and $1$
when $m$ is odd. 
\end{cor}

\section{\label{sec:Special-examples-in-SL}Special examples of spread models
in $SL_{m}\left(\mathbb{R}\right)$}

\subsection{Space of $d$-dimensional lattices in $\mathbb{R}^{n}$}

In a similar fashion to the construction of the space of full dimensional
lattices in $\mathbb{R}^{n}$, we can construct the space of unimodular
rank $d$ lattices with orientation in $\mathbb{R}^{n}$: 
\[
\latspace{d,n}:=\sl n\left(\RR\right)/T,
\]
where 
\[
T=\left[\begin{smallmatrix}\sl d\left(\ZZ\right) & \RR^{d,n-d}\\
0_{n-d,d} & \sl{n-d}\left(\RR\right)
\end{smallmatrix}\right]\times\left\{ \left[\begin{smallmatrix}\a^{-\frac{1}{d}}\idmat d & 0_{d,n-d}\\
0_{n-d,d} & \a^{\frac{1}{n-d}}\idmat{n-d}
\end{smallmatrix}\right]:\a>0\right\} .
\]
One can also get a fundamental domain for $\latspace{d,n}$ in $\sl n\left(\RR\right)$
using a variant of Iwasawa decomposition. For this denote by $P^{\prime\prime}$
the subgroup $\left[\begin{smallmatrix}P_{d} & 0_{d,n-d}\\
0_{n-d,d} & P_{n-d}
\end{smallmatrix}\right]$ in $\sl n\left(\RR\right)$, $K=\so n\left(\RR\right)$, $K^{\dprime}=\left[\begin{smallmatrix}\so d\left(\RR\right) & 0_{d,n-d}\\
0_{n-d,d} & \so{n-d}\left(\RR\right)
\end{smallmatrix}\right]$ and let $K^{\prime}$ be a spread model for $K/K^{\prime\prime}\diffeo\gras{d,n}$,
the Grassmanian of oriented $d$-dimensional subspaces of $\RR^{n}$
(see Proposition \ref{prop: K' is s.m. for sphere}). It is easy to
conclude that
\[
\latspace{d,n}=\sl n\left(\RR\right)/T\diffeo KP^{\dprime}/\left[\begin{smallmatrix}\sl d\left(\ZZ\right) & 0_{d,n-d}\\
0_{n-d,d} & \so{n-d}\left(\RR\right)
\end{smallmatrix}\right]
\]
 and that $K^{\prime}\left[\begin{smallmatrix}\groupfund d & 0_{d,n-d}\\
0_{n-d,d} & I_{n-d}
\end{smallmatrix}\right]$ is a set of representatives for $T$ inside $\sl n\left(\RR\right)$,
and for $\left[\begin{smallmatrix}\sl d\left(\ZZ\right) & 0_{d,n-d}\\
0_{n-d,d} & \so{n-d}\left(\RR\right)
\end{smallmatrix}\right]$ inside $KP^{\dprime}$. 

Next, we have the space of oriented normalized quotient latices of
$\Lambda/\Lambda_{d}$, where $\Lambda$ is a full lattice in $\RR^{n}$
and $\Lambda_{d}$ is a $d$-dimensional sub-lattice of $\Lambda$.
It is more convenient to identify $\Lambda/\Lambda_{d}$ with the
orthogonal projection of $\Lambda$ onto the subspace orthogonal to
$\sp{\RR}{\lat}$. Now it is easy to see that this space, say $\factor{\latspace{d,n}}$,
can be presented as the quotient 
\[
\factor{\latspace{d,n}}=\sl n\left(\RR\right)/T_{\#}=KP^{\dprime}/\left[\begin{smallmatrix}\so d\left(\RR\right) & 0_{d,n-d}\\
0_{n-d,d} & \sl{n-d}\left(\ZZ\right)
\end{smallmatrix}\right],
\]
 where 
\[
\factor T=\left[\begin{smallmatrix}\sl d\left(\RR\right) & \RR^{d,n-d}\\
0_{n-d,d} & \sl{n-d}\left(\ZZ\right)
\end{smallmatrix}\right]\times\left\{ \left[\begin{smallmatrix}\a^{-\frac{1}{d}}\idmat d & 0_{d,n-d}\\
0_{n-d,d} & \a^{\frac{1}{n-d}}\idmat{n-d}
\end{smallmatrix}\right]:\a>0\right\} .
\]
It is easy to see that $K^{\prime}\left[\begin{smallmatrix}I_{d} & 0_{d,n-d}\\
0_{n-d,d} & \groupfund{n-d}
\end{smallmatrix}\right]$ is a set of representatives for $\factor T$ inside $\sl n\left(\RR\right)$
and for $\left[\begin{smallmatrix}\so d\left(\RR\right) & 0_{d,n-d}\\
0_{n-d,d} & \sl{n-d}\left(\ZZ\right)
\end{smallmatrix}\right]$ inside $KP^{\dprime}$. 

In fact this space is diffeomorphic to $\latspace{n-d,n}$. To see
this write $g\in\sl n\left(\RR\right)$ as $g=\left(A|B\right)$,
where $A\in\RR^{n,d}$ and $B\in\RR^{n-d,n}$ and consider the map
from $\factor{\latspace{d,n}}$ to $\latspace{n-d,n}$ given by
\[
\left(A|B\right)\cdot\factor{T_{d}}\mapsto\left(\brac{I-A\brac{A^{\transpose}A}^{-1}A^{\transpose}}B|A\right)\cdot T_{n-d}
\]
and the inverse map is given by
\[
\brac{\hat{A}|\hat{B}}\cdot T_{n-d}\mapsto\left(\brac{I-\hat{A}\brac{\hat{A}^{T}\hat{A}}^{-1}\hat{A}^{\transpose}}\hat{B}|\hat{A}\right)\cdot\factor{T_{d}},
\]
where $\hat{B}\in\RR^{n,d}$ and $\hat{A}\in\RR^{n-d,n}$. We leave
it to the reader to check that these maps are well-defined and are
inverses of each other (see also the appendix of our forthcoming work
\cite{HK_dlattices}). 

A third and final space is the space of normalized pairs $\left(\Lambda,\factor{\Lambda}\right)$,
where $\Lambda$ is a rank $d$ oriented unimodular lattice and $\factor{\Lambda}$
is an oriented unimodular quotient lattice $\latfull/\Lambda$, with
$\latfull$ being a full lattice containing $\Lambda$. The resulting
space is modeled by 
\[
\pairspace{d,n}:=\sl n\left(\RR\right)/T_{d}^{\&}=KP^{\dprime}/\left[\begin{smallmatrix}\sl d\left(\ZZ\right) & 0_{d,n-d}\\
0_{n-d,d} & \sl{n-d}\left(\ZZ\right)
\end{smallmatrix}\right],
\]
where 
\[
T^{\&}=\left[\begin{smallmatrix}\sl d\left(\ZZ\right) & \RR^{d,n-d}\\
0_{n-d,d} & \sl{n-d}\left(\ZZ\right)
\end{smallmatrix}\right]\times\left\{ \left[\begin{smallmatrix}\a^{-\frac{1}{d}}\idmat d & 0_{d,n-d}\\
0_{n-d,d} & \a^{\frac{1}{n-d}}\idmat{n-d}
\end{smallmatrix}\right]:\a>0\right\} .
\]
It is easy to see that $K^{\prime}\left[\begin{smallmatrix}\groupfund d & 0_{d,n-d}\\
0_{n-d,d} & \groupfund{n-d}
\end{smallmatrix}\right]$ is a set of representatives for $T^{\&}$ inside $\sl n\left(\RR\right)$
and for $\left[\begin{smallmatrix}\sl d\left(\ZZ\right) & 0_{d,n-d}\\
0_{n-d,d} & \sl{n-d}\left(\ZZ\right)
\end{smallmatrix}\right]$ inside $KP^{\dprime}$.

\subsection{Spread models for lattice spaces}

The goal of this final part is to provide concrete examples for spread
models. These examples are all for spaces of lattices, which is where
the authors' interest in spread models stems from. 
\begin{prop}
\label{prop: spread models that we need}The following pairs consist
of quotient spaces and their spread models in the corresponding manifolds:
\begin{enumerate}
\item $\latspace m=\sl m\left(\RR\right)/\sl m\left(\ZZ\right)$, and $\groupfund m$
inside $\manifold=\sl m\left(\RR\right)$;
\item $\shapespace m=\so m\left(\RR\right)\backslash\sl m\left(\RR\right)/\sl m\left(\ZZ\right)\diffeo P_{m}/\sl m\left(\ZZ\right)$,
and $\symfund m$ inside $\manifold=P_{m}$.
\item $\pairspace{d,n}=KP^{\dprime}/\left[\begin{smallmatrix}\sl d\left(\ZZ\right) & 0_{d,n-d}\\
0_{n-d,d} & \sl{n-d}\left(\ZZ\right)
\end{smallmatrix}\right]$ and $K^{\prime}\left[\begin{smallmatrix}\groupfund d & 0_{d,n-d}\\
0_{n-d,d} & \groupfund{n-d}
\end{smallmatrix}\right]$ inside $KP^{\dprime}$
\item \label{enu: model space for n-1 dim lattices}$\latspace{d,n}=KP^{\dprime}/\left[\begin{smallmatrix}\sl d\left(\ZZ\right) & 0_{d,n-d}\\
0_{n-d,d} & \so{n-d}\left(\RR\right)
\end{smallmatrix}\right]$ and $K^{\prime}\left[\begin{smallmatrix}\groupfund d & 0_{d,n-d}\\
0_{n-d,d} & I_{n-d}
\end{smallmatrix}\right]$ inside $\manifold=KP^{\dprime}$.
\item $\factor{\latspace{d,n}}=KP^{\dprime}//\left[\begin{smallmatrix}\so d\left(\RR\right) & 0_{d,n-d}\\
0_{n-d,d} & \sl{n-d}\left(\ZZ\right)
\end{smallmatrix}\right]$ and $K^{\prime}\left[\begin{smallmatrix}I_{d} & 0_{d,n-d}\\
0_{n-d,d} & \groupfund{n-d}
\end{smallmatrix}\right]$ inside $\manifold=KP^{\dprime}$.
\end{enumerate}
\end{prop}

For the proof, we will use the following.
\begin{assumption}
\label{ass: half K can be BCS}According to Corollary \ref{cor:BCS fundamental domain for finite groups in compact groups},
we can choose all $K_{z}$'s to be BCS's. We choose the same BCS fundamental
domain $\mathcal{K}$ as $K_{z}$ for all $z\in P$ such that $\sym^{+}\left(z\right)$
iz Z$\left(K\right)$, and then apply Remark \ref{rem:spread models of subsubgroup}
to choose a BCS fundamental domain $K_{z}$ for the remaining $z$'s,
such that $K_{z}\subseteq\mathcal{K}$. 
\end{assumption}

\begin{proof}[Proof of Proposition \ref{prop: spread models that we need}]
\textbf{Parts 1 and 2.} The spaces $\latspace m$ and $\shapespace m$
are dealt similarly; since the quotients are by discrete subgroups,
we will verify that the fundamental domains are indeed spread models
using Remark \ref{rem:GlobalBCS}. For brevity, let $\mathcal{M}:=\so m\left(\RR\right)\backslash\sl m\left(\RR\right)=P_{m}$.
The boundary $\symfund m\cap\mathcal{M}_{free}$ is contained in a
finite union of lower dimensional submanifolds according to Corollary
\ref{cor:boundary of Fn}; for the boundary of $\groupfund m$, we
should use also the fact that all the $K$-fibers in the expression
for $\groupfund m$ given in Proposition \ref{prop: F of G'' from F of H}
are BCS's, by Assumption \ref{ass: half K can be BCS}. In more details,
\[
\groupfund m=(\mathcal{K}\cdot\brac{P_{m}\cap\interior{\symfund m}})\cup\bigcup_{i=1}^{\topindex\left(m\right)}K_{z_{i}}\cdot\left\{ z\in P_{m}\cap\del\symfund m:\sym^{+}\left(\lat_{z}\right)=\sym^{+}\left(\lat_{z_{i}}\right)\right\} 
\]
where $\left\{ z:\sym^{+}\left(\lat_{z}\right)=\sym^{+}\left(\lat_{z_{i}}\right)\right\} $
is contained in $\del\symfund m$,  and is therefore a BCS in $P_{m}$.
Since the fibers in $\so m\left(\RR\right)$ are BCS's due to Assumption
\ref{ass: half K can be BCS}, then by Fact \ref{rem: intersection and union of BCS}
and the fact that $\so m\left(\RR\right)\times P_{m}$ is homeomorphic
to $\sl m\left(\RR\right)$ we get that the boundary of $\groupfund m$
is also contained in a finite union of lower dimensional manifolds. 

The fact that the quotient map $\pi$ restricted to the closure is
proper is a consequence of the Mahler compactness criterion, as we
now explain. Assume that $B\subset\latspace m$ is compact, which
by Mahler's criterion means that there exists a positive constant
$\b$ such that for every $\lat\in B$, the length of the shortest
vector in $B$ is at least $\b$. Let $g=\pi\vert_{\overline{\groupfund m}}^{-1}\left(\lat\right)\in\overline{\groupfund m}$
and write $g=kan$ where $a=\diag{a_{1},\ldots,a_{m}}$. By construction
of $\groupfund m$, $a_{1}$ is the length of a shortest vector in
$\lat$, and therefore $\b\leq a_{1}$. Also by construction of $\groupfund m$,
the columns of $g$ are a reduced basis (Definition \ref{def: Siegel reduced basis})
to $\lat$; hence by part \ref{enu: entries of a increasing} Lemma
\ref{lem: BLC. facts about z in RS domain}, 
\[
0<\b\leq a_{1}\leq\brac{\sqrt{3}/2}a_{2}\leq\cdots\leq\brac{\sqrt{3}/2}^{m-1}a_{m}
\]
\[
\brac{\sqrt{3}/2}^{m-1}a_{m}=\brac{\sqrt{3}/2}^{m-1}\cdot\frac{1}{a_{1}\cdots a_{m-1}}\leq\brac{\sqrt{3}/2}^{\frac{\left(m-1\right)m}{2}}\cdot\frac{1}{\beta^{m-1}}.
\]
We conclude that $\left(a_{1},\ldots,a_{m-1}\right)$ lies in a bounded
subset of $\RR^{m-1}$, namely the $A$ components of  the elements
in $\pi\vert_{\overline{\groupfund m}}^{-1}\left(B\right)$ lie in
a compact set. The $N$ and $K$ components of the elements of $\groupfund m$
are bounded uniformly, so in particular it holds for the elements
of $\pi\vert_{\overline{\groupfund m}}^{-1}\left(B\right)$; we conclude
therefore that it is a compact set in $\groupfund m$, so the map
$\pi\vert_{\overline{\groupfund m}}$ is proper . The proof of properness
in the case of $\symfund m$ is identical. 

Finally, it is clear that $F_{m}\cap\mathcal{M}_{free}\subseteq\overline{\interior{F_{m}\cap\mathcal{M}_{free}}}$;
for $\widetilde{F_{m}}$, we have by Assumption \ref{ass: half K can be BCS}
that every $K_{z}$ lies in $\mathcal{K}$ and therefore,
\[
\overline{\interior{\widetilde{F_{m}}}}=\overline{\mathcal{K}\cdot\interior{F_{m}}}=\overline{\mathcal{K}}\cdot\overline{F_{m}}\supseteq\cup_{z\in\overline{F_{m}}}K_{z}\cdot z=\widetilde{F_{m}},
\]
where $\mathcal{K}$ is as in Assumption \ref{ass: half K can be BCS}.
Here the second equality follows from $\overline{\mathcal{K}}$ being
compact and the following inclusion is because $K_{z}\subset\mathcal{K}$. 

\textbf{Part 3.} Let us introduce the notation $G^{\dprime}$ for
the group $\left[\begin{smallmatrix}\sl d\left(\RR\right) & 0_{d,n-d}\\
0_{n-d,d} & \sl{n-d}\left(\RR\right)
\end{smallmatrix}\right]$. As for the space $\pairspace{d,n}$ we are going to use Proposition
\ref{prop: M/G * G/H =00003D M/H} with the space $KP^{\dprime}$,
the group $G^{\dprime}$ and its closed subgroup $\left[\begin{smallmatrix}\sl d\left(\ZZ\right) & 0_{d,n-d}\\
0_{n-d,d} & \sl{n-d}\left(\ZZ\right)
\end{smallmatrix}\right]$. Notice that $KP^{\dprime}/G^{\dprime}=K/K^{\dprime}$ so that, by
Proposition \ref{prop: S.M. for space that can be written as two quotients},
it follows that $K^{\prime}$ is a spread model in $KP^{\dprime}$
w.r.t.\ $G^{\dprime}$. By part 2 it is clear that $\left[\begin{smallmatrix}\groupfund d & 0_{d,n-d}\\
0_{n-d,d} & \groupfund{n-d}
\end{smallmatrix}\right]$ is a spread model in $G^{\dprime}$ w.r.t.\ $\left[\begin{smallmatrix}\sl d\left(\ZZ\right) & 0_{d,n-d}\\
0_{n-d,d} & \sl{n-d}\left(\ZZ\right)
\end{smallmatrix}\right]$. As a result, Proposition \ref{prop: M/G * G/H =00003D M/H} implies
that $K^{\prime}\left[\begin{smallmatrix}\groupfund d & 0_{d,n-d}\\
0_{n-d,d} & \groupfund{n-d}
\end{smallmatrix}\right]$ is indeed a spread model in $KP^{\dprime}$ for the group $\left[\begin{smallmatrix}\sl d\left(\ZZ\right) & 0_{d,n-d}\\
0_{n-d,d} & \sl{n-d}\left(\ZZ\right)
\end{smallmatrix}\right]$.

\textbf{Part 4 and 5.} These parts are proved in a similar fashion
as part 3.
\end{proof}
\bibliographystyle{alpha}
\phantomsection\addcontentsline{toc}{section}{\refname}\bibliography{bib_for_well_roundedness}

\end{document}